\documentclass[a4paper]{article} 
\usepackage{amscd, amsmath, amssymb, array, booktabs, braket, mathrsfs}
\usepackage{amsthm}
\usepackage[dvipdfm]{hyperref}
\usepackage[matrix,arrow]{xy}

\theoremstyle{definition}
\newtheorem{dfn}{Definition}[section]
\newtheorem{thm}[dfn]{Theorem}
\newtheorem{lem}[dfn]{Lemma}
\newtheorem{prp}[dfn]{Propostion}
\newtheorem{cor}[dfn]{Corollary}

\newtheorem{rem}[dfn]{Remark}
\newtheorem*{ack}{Acknowledgements}
\newtheorem*{plan}{Outline of the paper}
\newtheorem*{notation}{Notation}

\newtheorem{que}[dfn]{Question}

\newtheorem*{Gal}{Theorem A}
\newtheorem*{Liu1314}{Theorem B}

\newcommand{\bigslant}[2]{{\raisebox{.2em}{$#1$}\left/\raisebox{-.2em}{$#2$}\right.}}
\newcommand{\plim}[1][]{\mathop{\varprojlim}\limits_{#1}}
\newcommand{\RomanNumeralCaps}[1]{\MakeUppercase{\romannumeral #1}}

\begin{document}
\title{Stable $\mathbb{I}$-free lattices and the two-variable algebraic $p$-adic $L$-functions in residually reducible Hida deformation}
\author{Dong Yan}
\date{}
\maketitle
\begin{abstract}
This paper is a continuation of the author's previous work, where we studied the number of isomorphic classes of $G_{\mathbb{Q}}$-stable lattices in $p$-adic families of residually reducible ordinary modular Galois representations. In this paper, we study the number of isomorphic classes of $G_{\mathbb{Q}}$-stable free lattices in a residually reducible Hida deformation and the variation of their related two-variable algebraic $p$-adic $L$-functions. This result gives an answer of a question asked by Ochiai. 
\end{abstract}
\tableofcontents

\section{Introduction}
We fix an irregular pair $\left(p, k \right)$ i.e. $p$ is an irregular prime and $p$ divides the numerator of the $k$-th Bernoulli number $B_k$. By an idea of Ribet \cite{Ri76}, there is an eigen cusp form $f_k \in S_k\left(\mathrm{SL}_2\left(\mathbb{Z} \right) \right)$ such that the residual representation of the Galois representation $\rho_{f_k}$ attached to $f_k$ is reducible. This Galois representation $\rho_{f_k}$ and its Hida deformation play an important role in the Iwasawa theory for ideal class groups which is studied by Ribet \cite{Ri76}, Mazur-Wiles \cite{MW84} and Wile \cite{Wi90}. 

The author would like to study the Iwasawa theory for the two-variable Hida deformation of $f_k$ which is a special case of the Iwasawa theory for Galois deformations which is proposed by Greenberg \cite{RG}. One of subtle points in the residually reducible case is that the Selmer group and its characteristic ideal may depend on the choice of $G_{\mathbb{Q}}$-stable lattice. We first prepare some notation on Hida deformation. We fix embeddings $\overline{\mathbb{Q}} \hookrightarrow \overline{\mathbb{Q}_p}$ and $\overline{\mathbb{Q}} \hookrightarrow \mathbb{C}$, where $\overline{\mathbb{Q}}$ and $\overline{\mathbb{Q}_p}$ are the algebraic closures of the rational number field $\mathbb{Q}$ and the $p$-adic number field $\mathbb{Q}_p$ respectively. Let $\Gamma^{\prime}$ be the $p$-Sylow subgroup of group of diamond operators for the tower of modular curves $\left\{Y_1\left(p^t\right)\right\}_{t\geq 1}$ (see \S 3). For an integer $r$, we denote by $\mu_r$ the group of $r$-th roots of unity. Let $\omega: \left(\left.\mathbb{Z}\right/{p\mathbb{Z}} \right)^{\times} \rightarrow \mu_{p-1}$ be the Teichm\"uller character. In \cite{H2} and \cite{Wi88}, Hida and Wiles proved that there exists an integrally closed local domain $\mathbb{I}$ which is finite flat over $\Lambda:=\mathbb{Z}_p[[\Gamma^{\prime}]]$ and an $\mathbb{I}$-adic normalized eigen cusp form $\mathcal{F}=\displaystyle\sum_{n=1}^{\infty}a\left(n, \mathcal{F} \right)q^n\in\mathbb{I}[[q]]$ with character $\omega^{k-1}$ which can specialize to the $p$-stabilization of $f_k$ under an arithmetic specialization (see \S 3 below). 

Let $\mathbf{H}^{\mathrm{ord}}:=\mathbf{H}^{\mathrm{ord}}\left(1, \mathbb{Z}_p\right)$ be Hida's ordinary Hecke algebra over $\Lambda$ as in \cite{H1} and $\mathbf{h}^{\mathrm{ord}}:=\mathbf{h}^{\mathrm{ord}}\left(1, \mathbb{Z}_p\right)$ the quotient of $\mathbf{H}^{\mathrm{ord}}$ corresponding to cusp forms. By \cite[Theorem 3.1]{H1}, $\mathbf{H}^{\mathrm{ord}}$ and $\mathbf{h}^{\mathrm{ord}}$ are free $\Lambda$-modules of finite rank. The action of Hecke operators on the modular curve $X_1\left(p^r\right)$  induces convariant and cotravariant actions on its $p$-adic \'etale cohomology groups. We concern the covariant action throughout the paper. Let $\mathfrak{M}:=\mathfrak{M}\left(\omega^{k-2}, \mathbf{1} \right)$ be the Eisenstein maximal ideal of $\mathbf{H}^{\mathrm{ord}}$ as in \cite[\S 1.2, (1.2.9)]{MO05}. For simplicity, we assume the following condition throughout this section.
\begin{list}{}{}
\item[($\Lambda$)]The $\Lambda$-module $\mathbf{h}^{\mathrm{ord}}_{\mathfrak{M}}$ is of rank one over $\Lambda$. 
\end{list}
Note that the condition ($\Lambda$) holds for all irregular pairs $\left(p, k \right)$ with $p<10^7$ and $k<8000$ except for the pair $\left(p, k \right)=\left(547, 486 \right)$ (cf. \cite[Appendix II]{M11}). We also have that under the condition ($\Lambda$), the above $\mathbb{I}$-adic normalized eigen cusp form $\mathcal{F}$ has Fourier coefficients in $\Lambda$ i.e. $\mathbb{I}=\Lambda$ (cf. \cite[\S 7.6]{H3}).

In \cite[Theorem 2.1]{H2}, Hida constructed a Galois representation $\mathbb{V}_{\mathcal{F}}$ attached to $\mathcal{F}$ (see Theorem \ref{Hida} below). Let $\rho_{\mathcal{F}}$ be the action of $G_{\mathbb{Q}}$ on $\mathbb{V}_{\mathcal{F}}$. Recall that $\mathbb{V}_{\mathcal{F}}$ has a lattice $\mathbb{T}$ (see Definition \ref{lattice} below) which is stable under the action of $G_{\mathbb{Q}}$ (we call that $\mathbb{T}$ is a \textit{$G_{\mathbb{Q}}$-stable lattice} of $\mathbb{V}_{\mathcal{F}}$). The action of $G_{\mathbb{Q}}$ on $\mathbb{T}$ induced by $\rho_{\mathcal{F}}$ is continuous with respect to the $\mathfrak{m}_{\mathbb{I}}$-adic topology on $\mathrm{Aut}_{\mathbb{I}}\left(\mathbb{T} \right)$ and unramified outside $\set{p, \infty}$, where $\mathfrak{m}_{\mathbb{I}}$ denotes the maximal ideal of $\mathbb{I}$. Let us fix an integer $i$ from now on to the end of this paper. Let $\Gamma=\mathrm{Gal}\left(\mathbb{Q}_{\infty}/\mathbb{Q} \right)$, where $\mathbb{Q}_{\infty}$ is the cyclotomic $\mathbb{Z}_p$-extension of $\mathbb{Q}$. Let $\mathcal{R}=\mathbb{I}[[\Gamma]]$ and we introduce the $\mathcal{R}$-module $\mathcal{T}^{\left(i \right)}:=\mathbb{T}\hat{\otimes}_{\mathbb{Z}_p} \mathbb{Z}_p[[\Gamma]]\left(\tilde{\kappa}^{-1}\right)\otimes_{\mathbb{Z}_p}\omega^i$ on which $G_{\mathbb{Q}}$ acts diagonally, where $\tilde{\kappa}$ is the character $\tilde{\kappa}: G_{\mathbb{Q}} \twoheadrightarrow \Gamma \hookrightarrow \mathbb{Z}_p[[\Gamma]]^{\times}$. Recall that Wiles \cite[Theorem 2.2.2]{Wi88} proved that $\mathbb{V}_{\mathcal{F}}$ has a unique $D_p$-stable subspace $F^{+}\mathbb{V}_{\mathcal{F}}$ of dimension one such that the action of $D_p$ on $\left.\mathbb{V}_{\mathcal{F}}\right/F^{+}\mathbb{V}_{\mathcal{F}}$ is unramified. This induce an $\mathbb{I}[D_p]$-submodule $F^{+}\mathbb{T}=\mathbb{T} \cap F^{+}\mathbb{V}_{\mathcal{F}}$ of $\mathbb{T}$ and a $\mathcal{R}[D_p]$-submodule $F^{+}\mathcal{T}^{\left(i \right)}=F^{+}\mathbb{T}^{\left(i \right)} \hat{\otimes}_{\mathbb{Z}_p} \mathbb{Z}_p[[\Gamma]]\left(\tilde{\kappa}^{-1} \right)\otimes_{\mathbb{Z}_p}\omega^i$ of $\mathcal{T}^{\left(i \right)}$. For a $\mathcal{R}$-module $M$, we denote by $M^{\lor}=\mathrm{Hom}_{\mathbb{Z}_p}\left(M, \mathbb{Q}_p/\mathbb{Z}_p \right)$ the Pontrjagin dual of $M$. Let $\mathcal{A}=\mathcal{T}^{\left(i \right)}\otimes_{\mathcal{R}}\mathcal{R}^{\lor}$ and we denote by $F^{+}\mathcal{A}=F^{+}\mathcal{T}\otimes_{\mathcal{R}}\mathcal{R}^{\lor}$ in $\mathcal{A}$. Then we define the Selmer group $\mathrm{Sel}_{\mathcal{A}}$ as Definition \ref{7131}. 

In \cite[Proposition 4.9]{Ochiai06} and \cite[Remark 1.7 3-(b)]{Ochiai08}, Ochiai proved that $\left(\mathrm{Sel}_{\mathcal{A}}\right)^{\lor}$ is a finitely generated torsion $\mathcal{R}$-module for any $\mathbb{T}$. Then we denote by $\mathrm{char}_{\mathcal{R}}\left(\mathrm{Sel}_{\mathcal{A}} \right)^{\lor}$ the characteristic ideal of $\left(\mathrm{Sel}_{\mathcal{A}}\right)^{\lor}$. Let $L_p^{\mathrm{alg}}\left(\mathcal{T}^{\left(i \right)} \right)$ be a generator of $\left(\mathrm{Sel}_{\mathcal{A}}\right)^{\lor}$ which is well-defined up to multiplying by elements of $\mathcal{R}^{\times}$. We call $L_p^{\mathrm{alg}}\left(\mathcal{T}^{\left(i \right)} \right)$ \textit{the algebraic $p$-adic $L$-function for $\mathcal{T}^{\left(i \right)}$}.

Let $\rho_{\mathcal{F}}\left(\mathfrak{m}_{\mathbb{I}} \right)$ be the residual representation of $\rho_{\mathcal{F}}$ modulo $\mathfrak{m}_{\mathbb{I}}$. The existence and the uniqueness of such residual representation are proved by Mazur-Wiles \cite[\S 9]{MW2}. Note that $\rho_{\mathcal{F}}\left(\mathfrak{m}_{\mathbb{I}} \right)$ is isomorphic to the semi-simplification of the residual representation of $f_k$. Since $\rho_{\mathcal{F}}\left(\mathfrak{m}_{\mathbb{I}} \right)$ is reducible, the $G_{\mathbb{Q}}$-stable lattice $\mathbb{T}$ is not unique up to homothety. Thus $L_p^{\mathrm{alg}}\left(\mathcal{T}^{\left(i\right)} \right)$ may depend on $\mathbb{T}$. 

We call $\mathbb{T}$ a free lattice if $\mathbb{T}$ is a free $\mathbb{I}$-module. In general, a $G_{\mathbb{Q}}$-stable lattice $\mathbb{T}$ is not necessarily free over $\mathbb{I}$. However since $\mathbb{I}=\Lambda$ is a regular local ring, there exists a $G_{\mathbb{Q}}$-stable $\mathbb{I}$-free lattice by taking the double linear dual $\mathbb{T}^{**}:=\mathrm{Hom}_{\mathbb{I}}\left(\mathrm{Hom}_{\mathbb{I}}\left(\mathbb{T}, \mathbb{I}   \right), \mathbb{I}  \right)$ of a given $G_{\mathbb{Q}}$-stable lattice $\mathbb{T}$. We have a formula on the change of the algebraic $p$-adic $L$-functions for different choice of lattices due to Ochiai.

\begin{thm}[Ochiai {\cite[Corollary 4.4]{Ochiai08}}]\label{Ochiai thm}
Assume the condition ($\Lambda$). Let $\mathbb{T}$ and $\mathbb{T}^{\prime}$ be $G_{\mathbb{Q}}$-stable $\mathbb{I}$-free lattices of $\rho_{\mathcal{F}}$ such that $\mathbb{T}^{\prime} \subset \mathbb{T}$. Let $\mathcal{T}^{\left(i \right)} =\mathbb{T}\hat{\otimes}_{\mathbb{Z}_p} \mathbb{Z}_p[[\Gamma]]\left(\tilde{\kappa}^{-1}\right)\otimes_{\mathbb{Z}_p}\omega^i$ and ${\mathcal{T}^{\prime}}^{\left(i \right)}=\mathbb{T}^{\prime}\hat{\otimes}_{\mathbb{Z}_p} \mathbb{Z}_p[[\Gamma]]\left(\tilde{\kappa}^{-1}\right)\otimes_{\mathbb{Z}_p}\omega^i$. Then we have
\begin{equation*}
\dfrac{\left( L_p^{\mathrm{alg}}\left(\mathcal{T}^{\left(i \right)} \right)   \right)}{(L_p^{\mathrm{alg}}({\mathcal{T}^{\prime}}^{\left(i \right)} ) )}=\displaystyle\prod_{\mathfrak{P}\in P^1\left(\mathcal{R} \right)}\mathfrak{P}^{\mathrm{length}_{\mathcal{R}_{\mathfrak{P}}}\left(\left(\left.\mathcal{T}^{\left(i \right)}\right/{\mathcal{T}^{\prime}}^{\left(i \right)} \right)_{G_{\mathbb{R}}} \right)_{\mathfrak{P}}-\mathrm{length}_{\mathcal{R}_{\mathfrak{P}}}\left(\left(\left.F^{+}\mathcal{T}^{\left(i \right)}\right/F^{+}{\mathcal{T}^{\prime}}^{\left(i \right)} \right) \right)_{\mathfrak{P}}}, 
\end{equation*}
where $P^1\left(\mathcal{R} \right)$ is the set of all height-one primes of $\mathcal{R}$. 
\end{thm}

There is a ``geometric lattice" $\mathbb{T}_{\mathcal{F}}$ of $\mathbb{V}_{\mathcal{F}}$ which is constructed as follows
\begin{equation}\label{12022}
\mathbb{T}_{\mathcal{F}}:=\mathrm{Hom}_{\mathbb{I}}\left(\plim[r\geq1]H_{\text{\'et}}^1\left(X_1\left(p^r\right)\otimes_{\mathbb{Q}}\overline{\mathbb{Q}}, \mathbb{Z}_p \right)_{\mathfrak{M}}^{\mathrm{ord}}\otimes_{\mathbf{H}_{\mathfrak{M}}^{\mathrm{ord}}}\mathbb{I}, \mathbb{I}\right).
\end{equation}
Let $\mathcal{T}_{\mathcal{F}}^{\left(i \right)}:=\mathbb{T}_{\mathcal{F}}\hat{\otimes}_{\mathbb{Z}_p}\mathbb{Z}_p[[\Gamma]]\left(\tilde{\kappa}^{-1}\right)\otimes_{\mathbb{Z}_p}\omega^i$. As a continuation of the work on Iwasawa theory for two-variable residually reducible Hida deformations, Ochiai proposed the following question:
\begin{que}[Ochiai {\cite[Question 4.5 and \S 1]{Ochiai08}}]\label{q4.5}
\begin{enumerate}
\item[(1)] How many $G_{\mathbb{Q}}$-stable lattices up to $G_{\mathbb{Q}}$-isomorphism exist for a given Hida deformation (\cite[Question 4.5-(1)]{Ochiai08})?

\item[(2)] Can we calculate the variation of $L_p^{\mathrm{alg}}\left(\mathcal{T}^{\left(i\right)} \right)$ when $\mathbb{T}$ varies (\cite[Question 4.5-(2)]{Ochiai08})?

\item[(3)] Is $L_p^{\mathrm{alg}}(  \mathcal{T}_{\mathcal{F}}^{\left(i \right)}  )$ minimal under divisibility among the set of $L_p^{\mathrm{alg}}\left(\mathcal{T}^{\left(i\right)} \right)$ for all $G_{\mathbb{Q}}$-stable lattices of $\mathbb{V}_{\mathcal{F}}$ (\cite[\S 1 the statement below Remark 1.7]{Ochiai08})?
\end{enumerate}
\end{que}

Note that Question \ref{q4.5}-(3) is motivated by Stevens \cite{GS} in which we find a conjectural answer on the choice of the minimal lattice of the $p$-adic representation attached to an elliptic curve over $\mathbb{Q}$. It is conjectured by Greenberg that the $\mu$-invariant of the algebraic and the analytic $p$-adic $L$-functions for the minimal lattice are both zero. Then under the conjecture of $\mu=0$, the half of the main conjecture for residually reducible elliptic curves follows by Kato's theorem.

In this paper, we give an answer of the above question. First we give some remarks on Question \ref{q4.5}. In order to answer the question, first we must remark that we must make clear the question whether we consider only $G_{\mathbb{Q}}$-stable $\mathbb{I}$-free lattices or we consider all $G_{\mathbb{Q}}$-stable lattices. Recall that under certain conditions, we proved the following result for non-free lattices.

\begin{prp}[{\cite[Corollary 1.7]{Y}}]
Let $\left(p, k\right)$ be an irregular pair. Assume the condition ($\Lambda$) and that the ideal generated by Kubota-Leopoldt $p$-adic $L$-function $\mathcal{L}_p\left(\omega^{k-1}; \gamma^{\prime} \right) \in \Lambda$ (see \S 4.1 below) is a prime ideal of $\Lambda$. Suppose that $\mathcal{L}_p\left(\omega^{k-1}; \gamma^{\prime} \right)$ has a zero in $p\mathbb{Z}_p$, then there are infinitely many $G_{\mathbb{Q}}$-stable lattices of $\mathbb{V}_{\mathcal{F}}$ up to $G_{\mathbb{Q}}$-isomorphism. 
\end{prp}
However, by Lemma \ref{43} we know that to answer Question \ref{q4.5}-(2), it is enough to study the the variation of $L_p^{\mathrm{alg}}\left(\mathcal{T}^{\left(i \right)} \right)$ when $\mathbb{T}$ varies in the set of $G_{\mathbb{Q}}$-stable $\mathbb{I}$-free lattices. Our answers to the Question \ref{q4.5}-(1) (only for free lattices), (2) and (3) is the following theorem.

Let $\mathscr{L}^{\mathrm{fr}}\left(\rho_{\mathcal{F}} \right)$ be the set of isomorphic classes of $G_{\mathbb{Q}}$ stable $\mathbb{I}$-free lattices of $\mathbb{V}_{\mathcal{F}}$. We denote by $\mathscr{L}_p^{\mathrm{alg}}\left(\rho_{\mathcal{F}}^{\mathrm{n. ord}, \left(i \right)} \right)$ the set of all $\left(L_p^{\mathrm{alg}}\left(\mathcal{T}^{\left(i \right)} \right) \right)$ when $\mathbb{T}$ varies in the set of all  $G_{\mathbb{Q}}$-stable lattices of $\mathbb{V}_{\mathcal{F}}$. Our main result is as follows:

\begin{Gal}[Theorem \ref{yandongmain}, Corollary \ref{canonical}]\label{Gal}
Let $\left(p, k \right)$ be an irregular pair. Assume the condition ($\Lambda$), then we have the following statements:
\begin{enumerate}
\item[(1)]There are only finitely many $G_{\mathbb{Q}}$-stable $\mathbb{I}$-free lattices of $\mathbb{V}_{\mathcal{F}}$. Furthermore, we have
\begin{equation*}
\sharp\mathscr{L}^{\mathrm{fr}}\left(\rho_{\mathcal{F}} \right)=\displaystyle\prod_{\mathfrak{p} \in P^1\left(\mathbb{I} \right)}\left(\mathrm{ord}_{\mathfrak{p}}\mathcal{L}_p\left(\omega^{k-1}; \gamma^{\prime} \right)+1\right).
\end{equation*}
\item[(2)]Let $D\left(\mathcal{L}_p\left(\omega^{k-1}; \gamma^{\prime} \right) \right)$ be the set of all ideals of $\mathbb{I}$ which divide $\left(\mathcal{L}_p\left(\omega^{k-1}; \gamma^{\prime} \right)   \right)$. Then there exists a set $\mathscr{T}$ of a system of representatives of $\mathscr{L}^{\mathrm{fr}}\left(\rho_{\mathcal{F}}\right)$ which contains $\mathbb{T}_{\mathcal{F}}$ such that we have the following bijection
\begin{equation*}
\mathscr{T} \rightarrow D\left(\mathcal{L}_p\left(\omega^{k-1}; \gamma^{\prime} \right) \right), \mathbb{T} \mapsto \mathfrak{a}
\end{equation*}
such that $\left.\mathbb{T}\right/\mathbb{T}_{\mathcal{F}} \stackrel{\sim}{\rightarrow} \left.\mathbb{I}\right/\mathfrak{a}\,\left(\mathbf{1} \right)$, where $\left.\mathbb{I}\right/\mathfrak{a}\,\left(\mathbf{1} \right)$ denotes the $\mathbb{I}$-module $\left.\mathbb{I}\right/\mathfrak{a}$ on which $G_{\mathbb{Q}}$ acts trivially. 

\item[(3)]The algebraic $p$-adic $L$-function $L_p^{\mathrm{alg}}(  \mathcal{T}_{\mathcal{F}}^{\left(i \right)}  )$ is the minimal one under divisibility among the set of $L_p^{\mathrm{alg}}\left(\mathcal{T}^{\left(i\right)} \right)$ for all $G_{\mathbb{Q}}$-stable lattices of $\mathbb{V}_{\mathcal{F}}$. Furthermore, we have the following equality:
\begin{equation*}
\mathscr{L}_p^{\mathrm{alg}}\left(\rho_{\mathcal{F}}^{\mathrm{n.ord}, \left(i \right)} \right)=
\begin{cases}
\left\{ \left(L_p^{\mathrm{alg}}(  \mathcal{T}_{\mathcal{F}}^{\left(i \right)}  ) \right) \right\}& \left(i: \text{odd} \right) \\
L_p^{\mathrm{alg}}(  \mathcal{T}_{\mathcal{F}}^{\left(i \right)}  ) D\left(\mathcal{L}_p\left(\omega^{k-1}; \gamma^{\prime} \right) \right)& \left(i: \text{even} \right).
\end{cases}
\end{equation*}
\end{enumerate}
\end{Gal}

By Theorem A we have that to determine the set $\mathscr{L}_p^{\mathrm{alg}}\left(\rho_{\mathcal{F}}^{\mathrm{n. ord}, \left(i \right)} \right)$, it is enough to calculate $L_p^{\mathrm{alg}}(  \mathcal{T}_{\mathcal{F}}^{\left(i \right)}  )$. By combining Theorem A with a recent calculation of Bella\"iche and Pollack \cite{BP} on the one-variable cyclotomic deformation case, we obtain a result of $L_p^{\mathrm{alg}}(  \mathcal{T}_{\mathcal{F}}^{\left(0 \right)}  )$ as follows.

\begin{Liu1314}[Theorem \ref{910115}]
Let $\left(p, k \right)$ be an irregular pair. Assume ($\Lambda$) and the following condition
\begin{list}{}{}
\item[($p$Four)]We have the equality $\left(\mathcal{L}_p\left(\omega^{k-1}; \gamma^{\prime} \right) \right)=\left(a\left(p, \mathcal{F} \right)-1 \right)$ in $\mathbb{I}$.
\end{list}
Then $L_p^{\mathrm{alg}}(  \mathcal{T}_{\mathcal{F}}^{\left(0 \right)}  )$ is a unit of $\mathcal{R}$ and $$\mathscr{L}_p^{\mathrm{alg}}\left(\rho_{\mathcal{F}}^{\mathrm{n. ord}, \left(i \right)} \right)=D\left(\mathcal{L}_p\left(\omega^{k-1}; \gamma^{\prime} \right) \right).$$
\end{Liu1314}

\begin{plan}
In \S 2, we study some properties of $G$-stable lattices of the representation over a discrete valuation ring, which is used to prove Theorem A. In \S 3, we recall some known results on Hida deformation and Selmer group. In \S 4, we prove Theorem A and we give some examples of $\mathscr{L}^{\mathrm{fr}}\left(\rho_{\mathcal{F}} \right)$ and $\mathscr{L}_p^{\mathrm{alg}}\left(\rho_{\mathcal{F}}^{\mathrm{n. ord}, \left(i \right)} \right)$ in \S 5. By our arguments in \S 4, we can define a graph structure on the set of the isomorphic classes of $G_{\mathbb{Q}}$-stable $\mathbb{I}$-free lattices which we will study in the appendix.
\end{plan}

\begin{notation}
Let $R$ be a commutative domain and $K$ the field of fractions of $R$. For a finite dimensional $K$-vector space $V$ and a linear representation of a group $G$: $$\rho : G \rightarrow \mathrm{Aut}_{K}\left(V \right),$$ we denote by $\mathscr{C}\left(\rho \right)$ (resp. $\mathscr{C}^{\mathrm{fr}}\left(\rho \right)$) the set of homothetic classes of $G$-stable lattices (resp. $G$-stable $R$-free lattices) of $\rho$ and by $\mathscr{L}\left(\rho \right)$ (resp. $\mathscr{L}^{\mathrm{fr}}\left(\rho \right)$) the set of isomorphic classes of $G$-stable lattices (resp. $G$-stable $R$-free lattices) of $\rho$.

For a prime $l$, we denote by $I_l$ the inertia subgroup of the decomposition group $D_l$ at $l$. For a Dirichlet character $\theta$ modulo $M$, by abuse of notation, we sometimes denote by $\theta$ the character of $G_{\mathbb{Q}}$ composed with $G_{\mathbb{Q}} \twoheadrightarrow \mathrm{Gal}\left(\mathbb{Q}\left(\mu_{M} \right)/\mathbb{Q} \right) \stackrel{\sim}{\rightarrow} \left(\mathbb{Z}/M\mathbb{Z} \right)^{\times}$. We denote by $\kappa_{\mathrm{cyc}}$ (resp. $\kappa^{\prime}$) the isomorphism $\Gamma \stackrel{\sim}{\rightarrow} 1+p\mathbb{Z}_p$ (resp. $\Gamma^{\prime} \stackrel{\sim}{\rightarrow} 1+p\mathbb{Z}_p$). For later convenience, we choose $\gamma$ (resp. $\gamma^{\prime}$) a topological generator of $\Gamma$ (resp. $\Gamma^{\prime}$) such  that $\kappa_{\mathrm{cyc}}\left(\gamma\right)=\kappa^{\prime}\left(\gamma^{\prime}\right)=:u$. For an element $a \in \mathbb{Z}_p^{\times}$, write $a=\omega\left(a \right)\langle a \rangle$ under $\mathbb{Z}_p^{\times} \stackrel{\sim}{\rightarrow} \mu_{p-1}\times \left(1+p\mathbb{Z}_p\right)$ and we denote by $s_a$ the element of $\mathbb{Z}_p$ such that $\langle a \rangle=u^{s_a}$.

For a Dirichlet character $\theta$, write $\Lambda_{\theta}=\mathbb{Z}_p[\theta][[\Gamma^{\prime}]]$. For a fixed Noetherian integrally closed domain $R$, we denote by $P^{1}\left(R \right)$ the set of all height-one prime ideals of $R$. For a finitely generated $R$-module $M$, we denote by $M^{*}$ the $R$-linear dual of $M$ and by $M^{**}$ the double $R$-linear dual of $M$.

\end{notation}

\begin{ack}
The author expresses his sincere gratitude to Professor Tadashi Ochiai for spending a lot of time to read the manuscript carefully, giving the author useful comments and pointing out mistakes. He also thanks to Kenji Sakugawa for reading the manuscript, stimulating discussion and correcting several mistakes.
\end{ack}

\section{Stable lattices in the two-dimensional representation over a discrete valuation ring}
In this section, we study some properties of $G$-stable lattices of a representation over the field of fractions of a discrete valuation ring. The tool of counting the number of the isomorphic classes of stable lattices is the ideal of reducibility which is defined by Bella\"iche-Chenevier \cite{Bell2} for the residually reducible case. Although it might be known for the experts that the ideal of reducibility coincide with the whole ring if the residual representation is irreducible, we obtain the following proposition without the assumption on the residually reducibility. This will be used in \S 4 to verify whether the isomophic classes of $G_{\mathbb{Q}}$-stable $\mathbb{I}_{\mathfrak{p}}$-lattice in a Hida deformation for a height-one prime ideal $\mathfrak{p}$ is unique or not.
\begin{prp}\label{21}
Let $R$ be a local domain with maximal ideal $\mathfrak{m}$ and $K$ the field of fractions of $R$. We assume that the characteristic of the residue field $R/\mathfrak{m}$ is not two. Let $V$ be a vector space of dimension two over $K$ and $$\rho : G \rightarrow \mathrm{Aut}_K\left(V \right)$$ a linear representation of a group $G$ such that $\mathrm{tr}\rho\left(G \right) \subset R$. Assume that there exists an element $g_0 \in G$ such that the characteristic polynomial of $\rho\left(g_0\right)$ has roots in $R$ which are distinct modulo $\mathfrak{m}$. Then there exists a unique ideal $I\left(\rho \right)$ of $R$ such that for any ideal $J$, $I\left(\rho \right) \subset J$ if and only if there exist characters $\vartheta_1, \vartheta_2 : G \rightarrow \left(R/J \right)^{\times}$ such that $\mathrm{tr}\rho\ \mathrm{mod}\ J=\vartheta_1+\vartheta_2$. Furthermore, if $J \subset \mathfrak{m}$, the set of such characters $\set{\vartheta_1, \vartheta_2}$ is unique. 
\end{prp}

Proposition \ref{21} is proved in the same way as \cite[Lemme 1]{Bell2}. However, we add the proof for later reference.
\begin{proof}
The uniqueness of $I\left(\rho \right)$ follows by its property. We prove the existence. Let $\lambda_1$ and $\lambda_2$ be the roots of the characteristic polynomial of $\rho\left(g_0\right)$. Let us choose a $K$-basis $\set{e_1, e_2}$ of $V$ such that $\rho\left(g_0\right)=\begin{pmatrix} \lambda_1 & 0 \\ 0 & \lambda_2 \end{pmatrix}$. For any $g \in G$, write $\rho\left(g \right)=\begin{pmatrix} a\left(g \right) & b\left(g \right) \\ c\left(g \right) & d\left(g \right) \end{pmatrix}$ with respect to the basis $\set{e_1, e_2}$. Then we have the following equalities:
\begin{equation}\label{e0}
\begin{cases}
\mathrm{tr}\rho(g_0g)=\lambda_1a(g)+\lambda_2d(g) \in R \\
\mathrm{tr}\rho(g)=a(g)+d(g) \in R.
\end{cases}
\end{equation}
Since $R$ is a local domain and $\lambda_1 \not\equiv \lambda_2\ (\mathrm{mod}\ \mathfrak{m})$, $\lambda_1-\lambda_2$ is a unit of $R$. Thus $a(g), d(g) \in R$ by the equalities \eqref{e0}. Since $b(g)c(g^{\prime})=a(gg^{\prime})-a(g)a(g^{\prime})$, we have $b(g)c(g^{\prime}) \in R$ for any $g, g^{\prime} \in G$.

Let $I\left(\rho \right)$ be the $R$-submodule of $K$ which is generated by $b(g)c(g^{\prime})$ for all $g, g^{\prime} \in G$. Then $I\left(\rho \right)$ is an ideal of $R$. Let us take an ideal $J$ of $R$. The case when $J=R$ is obvious. Hence it is sufficient to consider the case when $J \subset \mathfrak{m}$. Assume $I\left(\rho \right) \subset J$. By the definition of $I\left(\rho \right)$, we have that $$a\,\mathrm{mod}\,I\left(\rho \right) : G \rightarrow \left(R/I\left(\rho \right) \right)^{\times}, g \mapsto a(g)\,\mathrm{mod}\,I\left(\rho \right)$$ and $$d\,\mathrm{mod}\,I\left(\rho \right) : G \rightarrow \left(R/I\left(\rho \right) \right)^{\times}, g \mapsto d(g)\,\mathrm{mod}\,I\left(\rho \right)$$are characters. We denote by $\vartheta_1$ (resp. $\vartheta_2$) : $G \rightarrow \left(R/J \right)^{\times}$ the composition of $a\,\mathrm{mod}\,I\left(\rho \right)$ (resp. $d\,\mathrm{mod}\,I\left(\rho \right)$) with the surjection $\left.R\right/I\left(\rho \right) \twoheadrightarrow R/J$. Then we have $\mathrm{tr}\rho\ \mathrm{mod}\ J=\vartheta_1+\vartheta_2$. 

For the converse, we denote by $\psi_i : G \rightarrow \left(R/\mathfrak{m} \right)^{\times}\ (i=1, 2)$ the composition of $\vartheta_i$ with the surjection $R/J \twoheadrightarrow R/\mathfrak{m}$. Then we have $\mathrm{tr}\,\rho\,\mathrm{mod}\,\mathfrak{m}=\psi_1+\psi_2$. Since for any $g \in G$, $$\mathrm{tr}\,\rho(g)^2-\mathrm{tr}\,\rho(g^2)=2\cdot\mathrm{det}\,\rho(g)$$ and $\mathrm{char}(R/\mathfrak{m}) \neq 2$, we have $\mathrm{det}\,\rho\,\mathrm{mod}\,\mathfrak{m}=\psi_1\psi_2$. Thus the mod\ $\mathfrak{m}$ characteristic polynomial of $\rho\left(g_0 \right)$ is $$\left(X-\psi_1\left(g_0 \right)\right)\left(X-\psi_2\left(g_0 \right)\right)=\left(X-\overline{\lambda}_1 \right)\left(X-\overline{\lambda}_2 \right),$$ where $\overline{\lambda}_i=\lambda_i\ \mathrm{mod}\ \mathfrak{m}\ (i=1, 2).$ Then we have $\set{\psi_1\left(g_0 \right), \psi_2\left(g_0 \right)}=\set{\overline{\lambda}_1, \overline{\lambda}_2}.$ Since $\lambda_1 \not\equiv \lambda_2 \ \left(\mathrm{mod}\ \mathfrak{m} \right)$, we have $\psi_1 \neq \psi_2$. Then by \cite[Lemme 1]{Bell2}, we have $\set{\vartheta_1, \vartheta_2}=\set{a\,\mathrm{mod}\,J, d\,\mathrm{mod}\,J}$ and $b\left(g \right)c\left(g^{\prime} \right) \in J$ for any $g, g^{\prime} \in G$. This implies $I\left(\rho \right) \subset J$. 

Now we prove the uniqueness of the set of characters $\set{\vartheta_1, \vartheta_2}$. First we prove the uniqueness of $\set{\psi_1, \psi_2}$ by contradiction. Assume we have another set of characters $\set{\psi^{\prime}_1, \psi^{\prime}_2}\neq\set{\psi_1, \psi_2}$ such that $\mathrm{tr}\,\rho\,\mathrm{mod}\,\mathfrak{m}=\psi^{\prime}_1+\psi^{\prime}_2$. By considering the $\mathrm{mod}\,\mathfrak{m}$ characteristic polynomial of $\rho\left(g \right)$, we have 
\begin{equation}\label{190224}
\psi^{\prime}_1\left(g \right)=\psi_1\left(g \right)\ \text{or}\ \psi^{\prime}_1\left(g \right)=\psi_2\left(g \right)
\end{equation}
for any $g \in G$. We may assume $\psi^{\prime}_1\left(g_0 \right)=\psi_1\left(g_0 \right)$ without loss of generality. Assume there exists an element $h \in G$ such that $\psi^{\prime}_1\left(h \right)=\psi_2\left(h \right)\neq\psi_1\left(h \right)$. Then we have
\begin{equation}\label{190419}
\psi^{\prime}_1\left(g_0 h \right)=\psi_1\left(g_0 \right)\psi_2\left(h \right).
\end{equation}
By \eqref{190224}, the equality \eqref{190419} contradicts to $\psi_1\left(g_0 \right) \neq \psi_2\left(g_0 \right)$ or $\psi_1\left(h \right) \neq \psi_2\left(h \right)$. Thus, under the assumption $\psi_1^{\prime}\left(g_0 \right)=\psi_1\left(g_0\right)$, we must have $\psi_1^{\prime}=\psi_1$ and $\psi_2^{\prime}=\psi_2$.

Now we complete the proof of Proposition \ref{22}. Suppose that we have another set of characters $\set{\vartheta_1^{\prime}, \vartheta_2^{\prime}}$ of $G$ with values in $\left(R/J \right)^{\times}$ such that $\mathrm{tr}\,\rho\,\mathrm{mod}\,J=\vartheta_1^{\prime}+\vartheta_2^{\prime}$. Then we have $\set{\vartheta^{\prime}_1\,\mathrm{mod}\,\mathfrak{m}, \vartheta^{\prime}_2\,\mathrm{mod}\,\mathfrak{m}}=\set{\psi_1, \psi_2}$ by the uniqueness of the set $\set{\psi_1, \psi_2}$, where $\vartheta^{\prime}_i\,\mathrm{mod}\,\mathfrak{m} : G \rightarrow \left(R/\mathfrak{m} \right)^{\times}\ (i=1, 2)$ is the composition of $\vartheta^{\prime}_i$ with the surjection $R/J \twoheadrightarrow R/\mathfrak{m}$. Then $\set{\vartheta_1^{\prime}, \vartheta_2^{\prime}}=\set{a\,\mathrm{mod}\,J, d\,\mathrm{mod}\,J}$ holds by \cite[Lemme 1]{Bell2}. Thus the uniqueness of $\set{\vartheta_1, \vartheta_2}$ follows. This completes the proof of Proposition \ref{21}.

\end{proof}

\begin{dfn}\label{152}
Let us keep the assumptions and the notation of Proposition \ref{21}. We call $I\left(\rho \right)$ the \textit{ideal of reducibility of $R$ corresponding to $\rho$} (cf. \cite[\S 2.3]{Bell2}). 
\end{dfn}

We recall the definition of lattice and stable lattice as follows:

\begin{dfn}\label{lattice}
Let $R$ be a commutative Noetherian integrally closed domain with field of fractions $K$. Let $V$ be a finite dimensional $K$-vector space. We say that a $R$-submodule $T$ of $V$ is a \textit{lattice} of $V$ if and only if $T$ is finitely generated and $T \otimes_R K=V$. Let $$\rho : G \rightarrow \mathrm{Aut}_{K}\left(V \right)$$ be a linear representation of a group $G$, we say that $T$ is a \textit{$G$-stable lattice of $V$} if $T$ is a lattice and $\rho\left(G \right)T=T$.
\end{dfn}

For the remainder of this section, let $R$ be a discrete valuation ring. We study the number of isomorphic classes of $G$-stable lattices by means of the ideal of reducibility. Firstly we recall the following proposition.

\begin{prp}[{\cite[Chap. 7, \S4.1, Corollary to Proposition 4]{Bour}}]\label{dvr1}
Let $A$ be a discrete valuation ring with $\varpi$ a fixed uniformizer and $K$ the field of fractions of $A$. Let $V$ be a finite dimensional $K$-vector space. We denote by $\hat{A}=\plim[j]A/\varpi^j A$ the $\varpi$-adic completion of $A$ and by $\hat{K}$ the field of fractions of $\hat{A}$. Let $\mathcal{C}_{A}$ (resp. $\mathcal{C}_{\hat{A}}$) be the category of $A$-lattices of $V$ (resp. the category of $\hat{A}$-lattices of $V\otimes_{K}\hat{K}$) and $F_1, F_2$ the following functors:$$F_1: \mathcal{C}_{A} \rightarrow \mathcal{C}_{\hat{A}}, T \mapsto T\otimes_A\hat{A},$$ $$F_2: \mathcal{C}_{\hat{A}} \rightarrow \mathcal{C}_{A}, \hat{T} \mapsto \hat{T}\cap V.$$Then we have $F_2\circ F_1=\mathrm{id}_{\mathcal{C}_{A}}$ and $F_1\circ F_2=\mathrm{id}_{\mathcal{C}_{\hat{A}}}$ i.e. the category $\mathcal{C}_{A}$ and $\mathcal{C}_{\hat{A}}$ are equivalent.

\end{prp}

\begin{prp}\label{22}
Let $A$ be a discrete valuation ring with $\varpi$ a fixed uniformizer and $K$ the filed of fractions of $A$. We assume the characteristic of the residue field $A/\left(\varpi \right)$ is not two. Let $V$ be a two-dimensional $K$-vector space and $$\rho : G \rightarrow \mathrm{Aut}_K\left(V \right)$$
a linear representation of a group $G$ such that $\rho$ has a $G$-stable lattice $T$. Assume that there exists an element $g_0 \in G$ such that the characteristic polynomial of $\rho\left(g_0\right)$ has roots in $A$ which are distinct modulo $\left(\varpi\right)$.  We denote by $I\left(\rho \right)$ the ideal of reducibility of $A$. Then we have the following statements:
\begin{enumerate}
\renewcommand{\labelenumi}{(\arabic{enumi})}
\item We have $I\left(\rho \right) \subset \left(\varpi \right)$ if and only if the semi-simplification $\left(T/\varpi T \right)^{\mathrm{ss}}$ is decomposed into two characters i.e. $\left(T/\varpi T \right)^{\mathrm{ss}} \cong A/\left(\varpi \right)(\psi_1) \oplus A/\left(\varpi \right)(\psi_2)$, where $\psi_i : G \rightarrow \left(A/\left(\varpi \right) \right)^{\times}$ ($i=1, 2$) is a character. Furthermore, we have $\psi_1 \neq \psi_2$ in this case. 
\item We have $\sharp\mathscr{C}\left(\rho \right)=\sharp\mathscr{L}\left(\rho \right)=\mathrm{ord}_{\varpi}I\left(\rho \right)+1$. 
\item More precisely, let $\mathrm{ord}_{\varpi}I\left(\rho \right)=n>0$. Then there exists a chain of $G$-stable lattices $T_n\supsetneq\cdots\supsetneq T_0$ which is a system of representatives of both $\mathscr{C}\left(\rho \right)$ and $\mathscr{L}\left(\rho \right)$ such that for any $1 \leq j \leq n$, $T_j/T_0$ is isomorphic to $A/\left(\varpi \right)^j$ as an $A$-module and $T_n/T_0$ is isomorphic to $A/\left(\varpi \right)^n\ (\vartheta_1^{\left(n \right)} )$ as an $A[G]$-module, where $\set{\vartheta_1^{\left(n \right)}, \vartheta_2^{\left(n \right)}}$ is the set of characters with values in $\left(\left.A\right/\left(\varpi \right)^n \right)^{\times}$ such that $\mathrm{tr}\,\rho\,\mathrm{mod}\,\left(\varpi \right)^n=\vartheta_1^{\left(n \right)}+\vartheta_2^{\left(n \right)}$.

\end{enumerate}
\end{prp}

\begin{proof}

We prove the first assertion.  By Proposition \ref{21}, the condition $I\left(\rho \right) \subset \left(\varpi \right)$ is equivalent to that there exist characters $\psi_1$ and $\psi_2$ of $G$ with values in $\left(\left.A\right/\left(\varpi \right) \right)^{\times}$ such that $\mathrm{tr}\,\rho\,\mathrm{mod}\,\left(\varpi \right)=\psi_1+\psi_2$. Hence it is equivalent to $\left(T/\varpi T \right)^{\mathrm{ss}} \cong A/\left(\varpi \right)(\psi_1) \oplus A/\left(\varpi \right)(\psi_2)$ by Brauer-Nesbitt theorem. We have $\psi_1\left(g_0 \right) \neq \psi_2\left(g_0 \right)$ under the assumption. This completes the proof of the first assertion.

We prove the second assertion. First we assume that $\left(T/\varpi T \right)^{\mathrm{ss}}$ is irreducible. Then every $G$-stable lattice is homothetic with $T$ by Nakayama's lemma. Thus $\sharp\mathscr{L}\left(\rho \right)\leq\sharp\mathscr{C}\left(\rho \right)=1$, and hence $\sharp\mathscr{L}\left(\rho \right)=1$. On the other hand the assumption that $\left(T/\varpi T \right)^{\mathrm{ss}}$ is irreducible is equivalent to $I\left(\rho \right)=A$ by (1). Thus the second assertion follows when $\left(T/\varpi T \right)^{\mathrm{ss}}$ is irreducible. Next we assume that $\left(T/\varpi T \right)^{\mathrm{ss}}$ is reducible. Then $\mathrm{tr}\,\rho\,\mathrm{mod}\,\left(\varpi \right)$ is the sum of two distinct characters by (1). We may assume that $A$ is complete by Proposition \ref{dvr1}. Then $\mathrm{ord}_{\varpi}I\left(\rho \right)+1=\sharp\mathscr{L}\left(\rho \right)=\sharp\mathscr{C}\left(\rho \right)$ follows by the same proof \cite[Proposition 3.4]{Y}, where the same statement is proved for the ring of integers of a finite extension of $\mathbb{Q}_p$. This completes the proof of the second assertion.

We prove the third assertion. By the same proof of \cite[Proposition 3.4]{Y}, we have that $T_j/T_0$ is isomorphic to $A/\left(\varpi \right)^j$ as an $A$-module for any $1\leq j \leq n$. Under the assumption $\mathrm{tr}\,\rho\,\mathrm{mod}\,\left(\varpi \right)^n=\vartheta_1^{\left(n \right)}+\vartheta_2^{\left(n \right)}$, we have that $\left.T_n\right/T_0$ is isomorphic to either $A/\left(\varpi \right)^n\ (\vartheta_1^{\left(n \right)} )$ or $A/\left(\varpi \right)^n\ (\vartheta_2^{\left(n \right)} )$ by Proposition \ref{21}. If $T_n/T_0 \stackrel{\sim}{\rightarrow} A/\left(\varpi \right)^n\ (\vartheta_2^{\left(n \right)} )$, one may change the chain $T_n\supsetneq\cdots\supsetneq T_0$ to $$T_{0} \supsetneq \varpi T_{1}\supsetneq\cdots\supsetneq\varpi^{n}T_n.$$This completes the proof of Proposition \ref{22}.

\end{proof}

\begin{rem}
Note that for all representation we consider in this paper, we always have $\sharp\mathscr{C}\left(\rho\right)=\sharp\mathscr{L}\left(\rho\right)$. However, they are different in general. For example, for the semi-simple representation $\mathbf{1}\oplus\chi_{\mathrm{cyc}}: G_{\mathbb{Q}} \rightarrow \mathrm{Aut}_{\mathbb{Z}_p}\left(\mathbb{Z}_p^{\oplus 2} \right)$, the tree $\mathscr{C}\left(\rho \right)$ is a full-line (cf. \cite[\S 17, Proposition 2]{M11}) and $\sharp\mathscr{C}\left(\rho \right)=\infty$. However let $T$ and $T^{\prime}$ be the representatives of the points $x$ and $x^{\prime}$ in $\mathscr{C}\left(\rho \right)$ respectively. Since the representation is semi-simple, $T$ is isomorphic to $T^{\prime}$ as $A[G]$-modules. Hence $\sharp\mathscr{L}\left(\rho \right)=1$. Since we consider the variation of the algebraic $p$-adic $L$-function when a stable lattice varies, we must minimize the possibility of the change of Selmer group. Thus we consider the isomorphic classes of stable lattices instead of homothetic classes. 

\end{rem}

\section{Hida deformation and Selmer group}
\subsection{Hida deformation}
We recall some known results on Galois representations attached to normalized $\mathbb{I}$-adic eigen cusp forms in this section. For more details, the reader can refer to \cite[Chapter 7]{H3}. 

We denote by $\Gamma_t^{\prime}$ the $p$-Sylow subgroup of the group of diamond operators on the modular curve $Y_1\left(p^{t+1}\right)$. $\Gamma_t^{\prime}$ is canonically isomorphic to the multiplicative group $\left.1+p\mathbb{Z}\right/1+p^{t+1}\mathbb{Z}$. We define $\Gamma^{\prime}:=\plim[t]\Gamma_t^{\prime}$. Recall that $\kappa^{\prime}$ denotes the canonical isomorphism $\kappa^{\prime}: \Gamma^{\prime} \stackrel{\sim}{\rightarrow} 1+p\mathbb{Z}_p$. 

A character $\nu$ of $\Gamma^{\prime}$ is called an \textit{arithmetic character of weight $k_{\nu}\in \mathbb{Z}_{\geq 2}$} if there exists an open subgroup $U$ of $\Gamma^{\prime}$ such that $\left.\nu\right|_{U}={\kappa^{\prime}}^{k_{\nu}-2}.$ Let $\mathbb{I}$ be an integrally closed local domain which is finite flat over $\Lambda$. We denote by $\mathfrak{X}_{\mathrm{arith}}\left(\mathbb{I}\right)$ the set of \textit{arithmetic specializations} which is the set of continuous homomorphisms defined as follows:$$\mathfrak{X}_{\mathrm{arith}}\left(\mathbb{I}\right)=\Set{\phi : \mathbb{I} \rightarrow \overline{\mathbb{Q}}_p | \text{$\left.\phi\right|_{\Gamma^{\prime}}$ is an arithmetic character of weight $k_{\phi} \in \mathbb{Z}_{\geq 2}$} }.$$For an arithmetic specialization $\phi$, we denote by $p^{r_{\phi}}$ the order of the character $\left.\phi\right|_{\Gamma^{\prime}}\cdot{\kappa^{\prime}}^{2-k_{\phi}}$ and by $\psi_{\phi}$ the following Dirichlet character 
\begin{equation*}
\psi_{\phi}: \left(\left.\mathbb{Z}\right/p^{r_{\phi}+1}\mathbb{Z}\right)^{\times} \twoheadrightarrow \left.1+p\mathbb{Z}\right/1+p^{r_{\phi}+1}\mathbb{Z} \stackrel{\sim}{\rightarrow} \bigslant{\Gamma^{\prime}}{{\Gamma^{\prime}}^{p^{r_{\phi}}}} \stackrel{\left.\phi\right|_{\Gamma^{\prime}}\cdot{\kappa^{\prime}}^{2-k_{\phi}}}{\longrightarrow} \overline{\mathbb{Q}_p}^{\times}.
\end{equation*}

\begin{dfn}\label{Hi1}
Let $\chi$ be a Dirichlet character modulo $Np$ with $\left(N, p \right)=1$ and $\mathbb{I}$ an integrally closed local domain which is finite flat over $\Lambda_{\chi}$. We call $\mathcal{F}=\displaystyle \sum^{\infty}_{n=1}a(n, \mathcal{F})q^n \in \mathbb{I}[[q]]$ an $\mathbb{I}$-adic normalized eigen cusp form with character $\chi$ if $$f_{\phi}:=\displaystyle \sum^{\infty}_{n=1}\phi(a(n, \mathcal{F}))q^n \in S_{k_{\phi}}\left(\Gamma_0(Np^{r_{\phi}+1}), \chi\psi_{\phi}\omega^{1-k_{\phi}}, \phi\left(\mathbb{I}\right) \right)$$ is a $p$-ordinary normalized eigen cusp form for all $\phi \in \mathfrak{X}_{\mathrm{arith}}\left(\mathbb{I}\right)$.
\end{dfn}

In \cite[Theorem 2.1]{H2}, Hida constructed a continuous Galois representation $\rho_{\mathcal{F}}$ attached to an $\mathbb{I}$-adic normalized eigen cusp form $\mathcal{F}$ as follows:
\begin{thm}[Hida {\cite[Theorem 2.1]{H2}}]\label{Hida}
Let $\mathcal{F}$ be an $\mathbb{I}$-adic normalized eigen cusp form with character $\chi$. Then there exist a $\mathbb{K}$-vector space $\mathbb{V}_{\mathcal{F}}$ of dimension two and a Galois representation $$\rho_{\mathcal{F}} : G_{\mathbb{Q}} \rightarrow \mathrm{Aut}_{\mathbb{K}}\left(\mathbb{V}_{\mathcal{F}} \right)$$ such that 
\begin{enumerate}
\renewcommand{\labelenumi}{(\arabic{enumi})}
\item There exists a $G_{\mathbb{Q}}$-stable lattice $\mathbb{T}$ of $\mathbb{V}_{\mathcal{F}}$ such the representation $G_{\mathbb{Q}} \rightarrow \mathrm{Aut}_{\mathbb{I}}\left(\mathbb{T} \right)$ induced by $\rho_{\mathcal{F}}$ is continuous with respect to the $\mathfrak{m}_{\mathbb{I}}$-adic topology on $\mathrm{Aut}_{\mathbb{I}}\left(\mathbb{T} \right)$.
\item The representation $\rho_{\mathcal{F}}$ is irreducible and unramified outside $Np\infty$.
\item For the geometric Frobenius element $\mathrm{Frob}_l$ at $l \nmid Np$, we have $$\mathrm{tr}\,\rho_{\mathcal{F}}\left(\mathrm{Frob}_l \right)=a\left(l, \mathcal{F} \right),$$ $$\mathrm{det}\,\rho_{\mathcal{F}}\left(\mathrm{Frob}_l \right)=\chi\left(l \right)\langle l \rangle{\kappa^{\prime}}^{-1}\left(\langle l \rangle     \right).$$
\end{enumerate}
\end{thm}
Although $\rho_{\mathcal{F}}$ may not have a $G_{\mathbb{Q}}$-stable $\mathbb{I}$-free lattice, we have the following proposition for the existence of the residual representation at a prime ideal of $\mathbb{I}$ (see \cite[\S 9]{MW2} for example).
\begin{prp}[Hida, Mazur-Wiles]\label{residual rep}
Let $\mathbb{V}_{\mathcal{F}}$ be the Galois representation attached to an $\mathbb{I}$-adic normalized eigen cusp form $\mathcal{F}$. Then for a prime ideal $\mathfrak{p}$ of $\mathbb{I}$, there exists a residual representation $$\rho_{\mathcal{F}}\left(\mathfrak{p} \right) : G_{\mathbb{Q}} \rightarrow \mathrm{GL}_2\left(\mathrm{Frac}\left(\mathbb{I}/\mathfrak{p} \right) \right)$$of $\rho_{\mathcal{F}}$ at $\mathfrak{p}$ such that $\rho_{\mathcal{F}}\left(\mathfrak{p} \right)$ is semi-simple, continuous under the $\mathfrak{m}_{\mathbb{I}}$-adic topology of $\mathrm{GL}_2\left(\mathrm{Frac}\left(\mathbb{I}/\mathfrak{p} \right) \right)$ and satisfies the following properties: 
\begin{enumerate}
\renewcommand{\labelenumi}{(\arabic{enumi})}
\item The representation $\rho_{\mathcal{F}}\left(\mathfrak{p} \right)$ is unramified outside $Np\infty$.
\item For the arithmetic Frobenius element $\mathrm{Frob}_l$ at $l \nmid Np$, we have
$$\mathrm{tr}\,\rho_{\mathcal{F}}(\mathfrak{p})(\mathrm{Frob}_l)=a\left(l, \mathcal{F} \right)\,\mathrm{mod}\,\mathfrak{p},$$
$$\mathrm{det}\,\rho_{\mathcal{F}}(\mathfrak{p})(\mathrm{Frob}_l)=\chi(l)\langle l \rangle{\kappa^{\prime}}^{-1}\left(\langle l \rangle     \right)\,\mathrm{mod}\,\mathfrak{p}.$$
\end{enumerate}
Furthermore, the residual representation $\rho_{\mathcal{F}}\left(\mathfrak{p} \right)$ is unique up to isomorphism over an algebraic closure of $\mathrm{Frac}(\mathbb{I}/\mathfrak{p})$. 
\end{prp}
As an representation of the decomposition subgroup $D_p$, we have the following property of $\rho_{\mathcal{F}}$ due to Mazur and Wiles: 
\begin{thm}[Wiles {\cite[Theorem 2.2.2]{Wi88}}]\label{W88}
Let $\mathbb{V}_{\mathcal{F}}$ be the Galois representation attached to an $\mathbb{I}$-adic normalized eigen cusp form $\mathcal{F}$. Then $\mathbb{V}_{\mathcal{F}}$ has a unique $D_p$-stable subspace $F^{+}\mathbb{V}_{\mathcal{F}}$ of dimension $1$ such that the action of $D_p$ on $\mathbb{V}_{\mathcal{F}}/F^{+}\mathbb{V}_{\mathcal{F}}$ is unramified.
\end{thm}

\subsection{Selmer groups for Galois deformations}

In this section, we recall the comparison formula of Selmer groups for two-variable Hida deformation which is proved by Ochiai \cite{Ochiai08}. First we recall the definition of Selmer group for a general Galois deformation. Let $\mathcal{R}$ be an integrally closed local domain which is finite flat over $\mathbb{Z}_p[[X_1, \cdots, X_n]]$ with $\mathcal{M}$ the maximal ideal of $\mathcal{R}$ and $\mathcal{K}$ the field of fractions. Let $\mathcal{V}$ be a finite-dimensional $\mathcal{K}$-vector space and $$\rho: G_{\mathbb{Q}} \rightarrow \mathrm{Aut}_{\mathcal{K}}\left(\mathcal{V} \right)$$ a linear representation such that 
\begin{enumerate}
\renewcommand{\labelenumi}{(\roman{enumi})}
\item The representation $\rho$ has a $G_{\mathbb{Q}}$-stable lattice $\mathcal{T}$.
\item The action of $G_{\mathbb{Q}}$ on $\mathcal{T}$ is continuous with respect to the $\mathcal{M}$-adic topology on $\mathrm{Aut}_{\mathcal{R}}\left(\mathcal{T} \right)$.
\item The the action of $G_{\mathbb{Q}}$ on $\mathcal{T}$ is unramified outside a finite set of primes $\Sigma \supset \set{p, \infty}$. 
\end{enumerate}

\begin{dfn}[{\cite[Definition 1.2]{Ochiai08}}]\label{7131}
Let $\mathcal{A}=\mathcal{T}\otimes_{\mathcal{R}}\mathcal{R}^{\lor}$
. Suppose that we have a $D_p$-stable subspace $F^{+}\mathcal{V}$ of $\mathcal{V}$. Let $F^{+}\mathcal{T}=F^{+}\mathcal{V} \cap \mathcal{T}$ and let $F^{+}\mathcal{A}=F^{+}\mathcal{T}\otimes_{\mathcal{R}}\mathcal{R}^{\lor}$ in $\mathcal{A}$. We define the Selmer group $\mathrm{Sel}_{\mathcal{A}}$ as follows:
$$\mathrm{Sel}_{\mathcal{A}}=\mathrm{Ker}\left[H^{1}\left(\mathbb{Q}_{\Sigma}/\mathbb{Q}, \mathcal{A} \right) \rightarrow \displaystyle\prod_{l \in \Sigma\setminus\{p, \infty\}}H^{1}\left(I_l, \mathcal{A} \right)\times H^{1}\left(I_p, \left.\mathcal{A}\right/F^{+}\mathcal{A} \right) \right].$$
\end{dfn}

The first well-known example for the above $\rho$ is the cyclotomic deformation of an ordinary $p$-adic representation. For simplicity, we only state the case of ordinary $p$-adic representation $V_f$ coming from a $p$-ordinary eigen cusp form $f$. Let $T$ be a $G_{\mathbb{Q}}$-stable $\mathcal{O}$-lattice of $V_f$, where $\mathcal{O}$ is the ring of integers of the field $\mathbb{Q}_p\left(\left\{a\left(n, f \right) \right\}_{n \geq 1} \right)$. Let $\widetilde{T}^{\left(i \right)}:=T\otimes_{\mathbb{Z}_p}\mathbb{Z}_p[[\Gamma]]\left(\tilde{\kappa}^{-1} \right)\otimes\omega^j$ where $G_{\mathbb{Q}}$ acts on diagonally. Let $\Lambda_{\mathcal{O}}^{\mathrm{cyc}}=\mathcal{O}[[\Gamma]]$ and $\widetilde{A}=\widetilde{T}^{\left(i \right)}\otimes_{\Lambda_{\mathcal{O}}^{\mathrm{cyc}}}{\Lambda_{\mathcal{O}}^{\mathrm{cyc}}}^{\lor}$. Since $f$ is $p$-ordinary, we have the $D_p$-stable subspace $F^{+}V_f \subset V_f$ of dimension one such that the action of $D_p$ on $\left.V_f\right/{F^{+}V_f}$ is unramified by Theorem \ref{W88}. Let $F^{+}T=T \cap F^{+}V_f$. We have $D_p$-stable submodules $F^{+}\widetilde{T}^{\left(i \right)}:=F^{+}T\otimes_{\mathbb{Z}_p}\mathbb{Z}_p[[\Gamma]]\left(\tilde{\kappa}^{-1} \right)\otimes\omega^i$ and $F^{+}\widetilde{A}:=F^{+}\widetilde{T}^{\left(i \right)}\otimes_{\Lambda_{\mathcal{O}}^{\mathrm{cyc}}}{\Lambda_{\mathcal{O}}^{\mathrm{cyc}}}^{\lor}$ of $\widetilde{T}^{\left(i \right)}$ and $\widetilde{A}$ respectively. Then one can define $\mathrm{Sel}_{\widetilde{A}}$ as Definition \ref{7131}. 
\begin{prp}[Greenberg {\cite[Proposition 3.2]{RG}}]\label{RGde}
We have that $\mathrm{Sel}_{\widetilde{A}}$ is isomorphic to $\mathrm{Sel}_{A}\left(\mathbb{Q}_{\infty} \right)$, where $A=T\otimes_{\mathbb{Z}_p}\otimes\left.\mathbb{Q}_p\right/\mathbb{Z}_p\otimes\omega^i$ and $\mathrm{Sel}_{A}\left(\mathbb{Q}_{\infty} \right) \subset H^{1}\left(\left.\mathbb{Q}_{\Sigma}\right/\mathbb{Q}_{\infty}, A \right)$ is the cyclotomic Selmer group defined in \cite{RG89}.
\end{prp}

By combining Proposition \ref{RGde} with Kato's theorem, we have that the Pontryagin dual $\mathrm{Sel}_{\widetilde{A}}^{\lor}$ is a finitely generated torsion $\Lambda_{\mathcal{O}}^{\mathrm{cyc}}$-module and 
\begin{equation*}
\mathrm{char}_{\Lambda_{\mathcal{O}}^{\mathrm{cyc}}}\left(\mathrm{Sel}_{\widetilde{A}}\right)^{\lor}=\mathrm{char}_{\Lambda_{\mathcal{O}}^{\mathrm{cyc}}}\mathrm{Sel}_{A}\left(\mathbb{Q}_{\infty} \right)^{\lor}.
\end{equation*}

Now let us consider the case where $\rho$ comes form a Hida deformation. Let us keep the notation of Hida deformation in the previous section. Recall that $\mathbb{V}_{\mathcal{F}}$ is the Galois representation attached to an $\mathbb{I}$-adic normalized eigen cusp form $\mathcal{F}$ and $\mathbb{T}$ a $G_{\mathbb{Q}}$-stable lattice of $\mathbb{V}_{\mathcal{F}}$. Let $\mathcal{T}^{\left(i \right)}=\mathbb{T}\hat{\otimes}_{\mathbb{Z}_p} \mathbb{Z}_p[[\Gamma]]\left(\tilde{\kappa}^{-1} \right)\otimes\omega^i$. Let $F^{+}\mathbb{T}=\mathbb{T}\cap F^{+}\mathbb{V}_{\mathcal{F}}$ and 
$F^{+}\mathcal{T}^{\left(i \right)}=F^{+}\mathbb{T}\hat{\otimes}_{\mathbb{Z}_p} \mathbb{Z}_p[[\Gamma]]\left(\tilde{\kappa}^{-1} \right)\otimes\omega^i$. Let $\mathbb{A}=\mathbb{T}^{\left(i \right)}\otimes_{\mathbb{I}}\mathbb{I}^{\lor}$ (resp. $\mathcal{A}=\mathcal{T}^{\left(i \right)}\otimes_{\mathbb{I}[[\Gamma]]}\mathbb{I}[[\Gamma]]^{\lor}$). Then one can define the Selmer group $\mathrm{Sel}_{\mathbb{A}}$ (resp. $\mathrm{Sel}_{\mathcal{A}}$) for one-variable (resp. two-variable) Hida deformation as Definition \ref{7131}. We have the following ``torsionness'' property for $\mathrm{Sel}_{\mathbb{A}}$ and $\mathrm{Sel}_{\mathcal{A}}$:
\begin{thm}[Ochiai {\cite[Corollary D]{Ochiai01}, \cite[Proposition 4.9]{Ochiai06}, \cite[Remark 1.7 3-(b)]{Ochiai08}}]\label{torsion}
Assume that the ring $\mathbb{I}$ is Gorenstein, then $\left(\mathrm{Sel}_{\mathbb{A}} \right)^{\lor}$ and $\left(\mathrm{Sel}_{\mathcal{A}} \right)^{\lor}$ are finitely generated torsion modules over $\mathbb{I}$ and $\mathbb{I}[[\Gamma]]$ respectively.
\end{thm}

Since we are interested in the residually reducible case, it is important to calculate the difference of Selmer groups for different choice of $G_{\mathbb{Q}}$-stable lattice. In \cite{Ochiai08}, by generalizing the method of Perrin-Riou \cite{PR}, Ochiai gived a formula on calculation of the difference of Selmer groups for Galois deformation. Although Ochiai's theorem could apply to plenty Galois deformations, we only state the case of the two-variable Hida deformation which will be used in the remainder of this paper.

\begin{thm}[Ochiai {\cite[Theorem 1.6 and Corollary 4.4]{Ochiai08}}]\label{O}
Let $\mathbb{V}_{\mathcal{F}}$ be the Galois representation attached to an $\mathbb{I}$-adic normalized eigen cusp form $\mathcal{F}$. Suppose that $\mathbb{I}$ is isomorphic to $\mathcal{O}[[X]]$, where $\mathcal{O}$ is the ring of integers of a finite extension of $\mathbb{Q}_p$. Let $\mathbb{T}$ and $\mathbb{T}^{\prime}$ be $G_{\mathbb{Q}}$-stable $\mathbb{I}$-free lattices of $\rho_{\mathcal{F}}$. Let $\mathcal{T}^{\left(i \right)} =\mathbb{T}\hat{\otimes}_{\mathbb{Z}_p} \mathbb{Z}_p[[\Gamma]]\left(\tilde{\kappa}^{-1}\right)\otimes_{\mathbb{Z}_p}\omega^i$ (resp. ${\mathcal{T}^{\prime}}^{\left(i \right)}=\mathbb{T}^{\prime}\hat{\otimes}_{\mathbb{Z}_p} \mathbb{Z}_p[[\Gamma]]\left(\tilde{\kappa}^{-1}\right)\otimes_{\mathbb{Z}_p}\omega^i$) and $\mathcal{A}=\mathcal{T}^{\left(i \right)}\otimes_{\mathbb{I}[[\Gamma]]}\mathbb{I}[[\Gamma]]^{\lor}$ (resp. $\mathcal{A}^{\prime}={\mathcal{T}^{\prime}}^{\left(i \right)}\otimes_{\mathbb{I}[[\Gamma]]}\mathbb{I}[[\Gamma]]^{\lor}$). Then we have the following equality:
\begin{equation*}
\dfrac{\mathrm{char}_{\mathbb{I}[[\Gamma]]}\left(\mathrm{Sel}_{\mathcal{A}}\right)^{\lor}}{\mathrm{char}_{\mathbb{I}[[\Gamma]]}\left(\mathrm{Sel}_{\mathcal{A}^{\prime}}\right)^{\lor}}=\displaystyle\prod_{\mathfrak{P}\in P^1\left(\mathbb{I}[[\Gamma]] \right)}\mathfrak{P}^{\mathrm{length}_{\mathbb{I}[[\Gamma]]_{\mathfrak{P}}}\left(\left(\left.\mathcal{T}^{\left(i \right)}\right/{\mathcal{T}^{\prime}}^{\left(i \right)} \right)_{G_{\mathbb{R}}} \right)_{\mathfrak{P}}-\mathrm{length}_{\mathbb{I}[[\Gamma]]_{\mathfrak{P}}}\left(\left(\left.F^{+}\mathcal{T}^{\left(i \right)}\right/F^{+}{\mathcal{T}^{\prime}}^{\left(i \right)} \right) \right)_{\mathfrak{P}}}.
\end{equation*}
\end{thm}

\section{Proof of the main theorem}

\subsection{Statement of the main result}

Let us keep the notation of the previous section. In this subsection, we state Theorem A in more general settings. We fix a pair $\left(N, \chi \right)$ where $N$ is a positive integer with $\left(N, p \right)=1$ and $\chi$ a Dirichlet character modulo $Np$ form now on to the end of this paper. Recall that Iwasawa \cite{Iw68} showed that there exists a unique power series $\mathcal{L}_p\left(\chi; \gamma^{\prime} \right) \in \Lambda_{\chi}$ such that 
\begin{eqnarray*}
\phi\left(\mathcal{L}_p\left(\chi; \gamma^{\prime} \right) \right) =\left\{ \begin{array}{ll}
L_p\left(1-k_{\phi}, \chi\psi_{\phi}\omega \right) & (\chi \neq \omega^{-1}) \\
\left(\psi_{\phi}\left(u \right)u^{k_{\phi}}-1 \right)L_p\left(1-k_{\phi}, \psi_{\phi} \right) & (\chi=\omega^{-1})\\
\end{array} \right.
\end{eqnarray*}
for any arithmetic specialization $\phi \in \mathfrak{X}_{\mathrm{arith}}\left(\Lambda_{\theta}\right)$. Note that we modify $\mathcal{L}_p\left(\chi; \gamma^{\prime} \right)$ slightly different from Iwasawa's original paper \cite{Iw68} for the compatibility of the arithmetic specialization.

Let $\mathbb{V}_{\mathcal{F}}$ be the Galois representation attached to an $\mathbb{I}$-adic normalized eigen cusp form $\mathcal{F}$ with character $\chi$ which is fixed in the rest of the paper. We always assume the following condition:
\begin{list}{}{}
\item[($D_p$-dist)]The residual representation $\rho_{\mathcal{F}}\left(\mathfrak{m}_{\mathbb{I}} \right)$ is $D_p$-distinguished i.e. $\left.\rho_{\mathcal{F}}\left(\mathfrak{m}_{\mathbb{I}} \right)\right|_{D_p}$ is decomposed into two distinct characters of $D_p$ with values in $\left(\mathbb{I}/\mathfrak{m}_{\mathbb{I}} \right)^{\times}$. 
\end{list}
Since we are interested in the case when $\rho_{\mathcal{F}}\left(\mathfrak{m}_{\mathbb{I}} \right)$ is reducible, we assume the following condition from now on to the end of this paper
\begin{list}{}{}
\item[(Red)]The residual representation $\rho_{\mathcal{F}}\left(\mathfrak{m}_{\mathbb{I}} \right)$ is isomorphic to $\mathbf{1}\oplus\overline{\chi}$, where $\overline{\chi}$ is the character 
$\overline{\chi}: G_{\mathbb{Q}} \stackrel{\chi}{\rightarrow} \mathbb{I}^{\times} \twoheadrightarrow \left(\left.\mathbb{I}\right/\mathfrak{m}_{\mathbb{I}} \right)^{\times}.$
\end{list}

Let $\varepsilon_1\ (\mathrm{resp}.\ \varepsilon_2) : D_p \rightarrow \mathbb{I}^{\times}$ be the character such that $D_p$ acts on $F^{+}\mathbb{V}_{\mathcal{F}}$ (resp. $\mathbb{V}_{\mathcal{F}}/F^{+}\mathbb{V}_{\mathcal{F}}$) via $\varepsilon_1$ (resp. $\varepsilon_2$). Since $\varepsilon_2$ is unramified, under the conditions ($D_p$-dist) and (Red) we have $\overline{\varepsilon_1}=\left.\overline{\chi}\right|_{D_p}$ and $\overline{\varepsilon_2}=\mathbf{1}$. Let $\varphi$ the Euler totient function. Write $\mathcal{R}=\mathbb{I}[[\Gamma]]$ from now on to the end of this paper. Recall that $\mathscr{L}^{\mathrm{fr}}\left(\rho_{\mathcal{F}} \right)$ is the set of isomorphic classes of $G_{\mathbb{Q}}$-stable $\mathbb{I}$-free lattices of $\mathbb{V}_{\mathcal{F}}$. Recall that $i$ is a fixed integer throughout the paper. For a $G_{\mathbb{Q}}$-stable lattice $\mathbb{T}$ of $\mathbb{V}_{\mathcal{F}}$, let $\mathcal{T}^{\left(i \right)}=\mathbb{T}\hat{\otimes}_{\mathbb{Z}_p}\mathbb{Z}_p[[\Gamma]]\left(\tilde{\kappa}^{-1} \right)\otimes\omega^i$ and $\mathcal{A}=\mathcal{T}^{\left(i \right)}\otimes_{\mathcal{R}}\mathcal{R}^{\lor}$. Let $L_p^{\mathrm{alg}}\left(  \mathcal{T}^{\left(i \right)} \right)$ be a generator of $\mathrm{char}_{\mathcal{R}}\left(\mathrm{Sel}_{\mathcal{A}}\right)^{\lor}$, then $L_p^{\mathrm{alg}}\left(  \mathcal{T}^{\left(i \right)} \right)$ is well-defined up to multiplications by elements of $\mathcal{R}^{\times}$. Let $\mathscr{L}_p^{\mathrm{alg}}\left(\rho_{\mathcal{F}}^{\mathrm{n.ord}, \left(i \right)} \right)$ be the set of all $\left(L_p^{\mathrm{alg}}\left(  \mathcal{T}^{\left(i \right)} \right) \right)$ when $\mathbb{T}$ varies. Our main result in this paper is the following theorem:

\begin{thm}\label{yandongmain}
Suppose $\mathbb{I}$ is isomorphic to $\mathcal{O}[[X]]$. Assume $p \nmid \varphi(N)$ and the conditions (Red), ($D_p$-dist). Let $\mathbb{J}$ (resp. $\mathcal{J}$) be the ideal of $\mathbb{I}$ (resp. $\mathcal{R}$) which is generated by $a\left(l, \mathcal{F} \right)-1-\chi\left(l \right)\langle l \rangle{\kappa^{\prime}}^{-1}\left(\langle l \rangle \right)$ for all primes $l \nmid Np$ and $a\left(p, \mathcal{F} \right)-1$.  Then we have the following statements:
\begin{enumerate}
\item[(1)] The set $\mathscr{L}^{\mathrm{fr}}\left(\rho_{\mathcal{F}} \right)$ is finite and we have 
\begin{equation*}\label{12101}
\sharp\mathscr{L}^{\mathrm{fr}}\left(\rho_{\mathcal{F}} \right)=\sharp\mathscr{C}^{\mathrm{fr}}\left(\rho_{\mathcal{F}} \right)=\displaystyle\prod_{\mathfrak{p} \in P^1\left(\mathbb{I} \right)}\left(\mathrm{ord}_{\mathfrak{p}}\mathbb{J}_{\mathfrak{p}} +1 \right)\leq\displaystyle\prod_{\mathfrak{p} \in P^1\left(\mathbb{I} \right)}\left(\mathrm{ord}_{\mathfrak{p}}\mathcal{L}_p\left(\chi; \gamma^{\prime} \right)+1\right).
\end{equation*}
\item[(2)]There exist a set $\mathscr{T}$ of a system of representatives of $\mathscr{C}^{\mathrm{fr}}\left(\rho_{\mathcal{F}}\right)$ and a lattice $\mathbb{T}^{\mathrm{min}} \in \mathscr{T}$ such that we have the following bijection
\begin{equation*}
\mathscr{T} \rightarrow D\left(\mathbb{J}^{**} \right), \mathbb{T} \mapsto \mathfrak{a}
\end{equation*}
such that $\left.\mathbb{T}\right/\mathbb{T}^{\mathrm{min}} \stackrel{\sim}{\rightarrow} \left.\mathbb{I}\right/\mathfrak{a}\,\left(\mathbf{1} \right).$
\item[(3)]Let $\mathbb{T}^{\mathrm{min}}$ be the $G_{\mathbb{Q}}$-stable $\mathbb{I}$-free lattice in the second assertion. We have the following equalities:
\begin{equation*}
\mathscr{L}_p^{\mathrm{alg}}\left(\rho_{\mathcal{F}}^{\mathrm{n.ord}, \left(i \right)} \right)=
\begin{cases}
\left\{ \left(L_p^{\mathrm{alg}}\left(\mathcal{T}^{\mathrm{min}, \left(i \right)} \right)\right) \right\}& \left(i: \text{odd} \right) \\
L_p^{\mathrm{alg}}\left(\mathcal{T}^{\mathrm{min}, \left(i \right)} \right) D\left(\mathcal{J}^{**} \right)& \left(i: \text{even} \right).
\end{cases}
\end{equation*}
\end{enumerate}
\end{thm}

\subsection{Preparation for the proof of Theorem \ref{yandongmain}}
Before we prove Theorem \ref{yandongmain}, we obtain some results on Hida deformation as preparation. By the following lemma, we know that it is enough to study the variation of $L_p^{\mathrm{alg}}\left(\mathcal{T}^{\left(i \right)} \right)$ when $\mathbb{T}$ varies in the set of $\mathbb{I}$-free lattices. 
\begin{lem}[{\cite[Lemma 4.3]{Ochiai08}}]\label{43}
Let $\mathbb{T}$ be a $G_{\mathbb{Q}}$-stable lattice of $\mathbb{V}_{\mathcal{F}}$. Let $\mathcal{T}^{\left(i \right)}=\mathbb{T}\hat\otimes_{\mathbb{Z}_p} \mathbb{Z}_p[[\Gamma]]\left(\tilde{\kappa}^{-1} \right)\otimes\omega^i$ and ${\mathcal{T}^{**}}^{\left(i \right)}=\mathbb{T}^{**}\hat\otimes_{\mathbb{Z}_p} \mathbb{Z}_p[[\Gamma]]\left(\tilde{\kappa}^{-1} \right)\otimes\omega^i$. Assume that $\mathbb{I}$ is a regular local ring. Then we have the following equality 
\begin{equation*}
\left(L_p^{\mathrm{alg}}\left(\mathcal{T}^{\left(i \right)} \right) \right)=\left(L_p^{\mathrm{alg}}\left({\mathcal{T}^{**}}^{\left(i \right)} \right)    \right)
\end{equation*}
and ${\mathcal{T}^{**}}^{\left(i \right)}$ is free over $\mathcal{R}$.
\end{lem}

Suppose that $\mathbb{V}_{\mathcal{F}}$ has a $G_{\mathbb{Q}}$-stable $\mathbb{I}$-free lattice $\mathbb{T}$. Let $F^{-}\mathbb{T}=\mathbb{T}/F^{+}\mathbb{T}$. Note that $F^{+}\mathbb{T}$ and $F^{-}\mathbb{T}$ may not be free of rank one over $\mathbb{I}$. However, we have the following lemma which is proved by Fouquet and Ochiai \cite{FO12}:

\begin{lem}[Fouquet-Ochiai {\cite[Remark 2.13-(3)]{FO12}}]\label{FO}
Suppose that $\mathbb{V}_{\mathcal{F}}$ has a $G_{\mathbb{Q}}$-stable $\mathbb{I}$-free lattice $\mathbb{T}$. Assume the condition ($D_p$-dist), then $F^{+}\mathbb{T}$ and $F^{-}\mathbb{T}$ are free $\mathbb{I}$-modules of rank one.

\end{lem}

Recall that $\gamma$ (resp. $\gamma^{\prime}$) is a topological generator of $\Gamma$ (resp. $\Gamma^{\prime}$) and $u$ is a topological generator of $1+p\mathbb{Z}_p$ such that $$\kappa_{\mathrm{cyc}}\left(\gamma\right)=\kappa^{\prime}\left(\gamma^{\prime}\right)=u.$$We denote by $\kappa_{\mathrm{cyc}}^{\mathrm{univ}}$ the  following character of $G_{\mathbb{Q}}$:
\begin{equation*}
\kappa_{\mathrm{cyc}}^{\mathrm{univ}} : G_{\mathbb{Q}} \twoheadrightarrow \Gamma \stackrel{\kappa_{\mathrm{cyc}}}{\rightarrow} 1+p\mathbb{Z}_p \stackrel{{\kappa^{\prime}}^{-1}}{\rightarrow} \Gamma^{\prime} \hookrightarrow \Lambda^{\times}.
\end{equation*}
Let $I\left(\rho_{\mathcal{F}} \right)$ be the ideal of reducibility of $\mathbb{I}$ (cf. Definition \ref{152}). We have the following lemma:
\begin{lem}\label{3.7}
Let us keep the assumptions and the notation of Theorem \ref{yandongmain}. Then under the assumption $p \nmid \varphi\left(N \right)$, we have $$\mathrm{tr}\rho_{\mathcal{F}}\equiv\chi\kappa_{\mathrm{cyc}}\kappa_{\mathrm{cyc}}^{\mathrm{univ}}+\mathbf{1}\ \left(\mathrm{mod}\ I\left(\rho_{\mathcal{F}} \right) \right)$$ and $I\left(\rho_{\mathcal{F}} \right)=\mathbb{J}$.
\end{lem}
\begin{proof}
Under the assumption ($D_p$-dist), one could choose an element $g_0 \in D_p$ such that $\overline{\varepsilon}_1(g_0) \neq \overline{\varepsilon}_2(g_0)$. Let $\left\{e_1, e_2 \right\}$ be a basis of $\mathbb{V}_{\mathcal{F}}$ such that 
\begin{equation}\label{0118}
\rho_{\mathcal{F}}(g_0) = \begin{pmatrix} \varepsilon_1\left(g_0 \right) & 0 \\ 0 & \varepsilon_2\left(g_0 \right) \end{pmatrix}\ \text{and}\ \rho_{\mathcal{F}}\mid_{D_p}= \begin{pmatrix} \varepsilon_1 & * \\ 0& \varepsilon_2 \end{pmatrix}.
\end{equation}
For any $g \in G_{\mathbb{Q}}$, write $\rho_{\mathcal{F}}(g) = \begin{pmatrix} a(g) & b(g) \\ c(g) & d(g) \end{pmatrix}$. We have that $\mathrm{tr}\,\rho_{\mathcal{F}}\,\mathrm{mod}\,\mathbb{J}$ is the sum of two characters by Chebotarev density theorem. Hence $I\left(\rho_{\mathcal{F}} \right) \subset \mathbb{J}$ by Proposition \ref{21}. Recall that under the assumptions (Red) and ($D_p$-dist), we have $\overline{\varepsilon}_2=\mathbf{1}$. Then under the assumption $p \nmid \varphi\left(N \right)$, we have $d\left(g \right) \equiv 1\left(g \right)\ \left(\mathrm{mod}\ I\left(\rho_{\mathcal{F}} \right) \right)$ for all $g \in G_{\mathbb{Q}}$ by class field theory (cf. \cite[Lemma 3.7]{Y}). Hence $a\left(g \right) \equiv \chi\kappa_{\mathrm{cyc}}\kappa_{\mathrm{cyc}}^{\mathrm{univ}}\left(g \right)\ \left(\mathrm{mod}\ I\left(\rho_{\mathcal{F}} \right) \right)$ since the determinant $\mathrm{det}\,\rho_{\mathcal{F}}=\chi\kappa_{\mathrm{cyc}}\kappa_{\mathrm{cyc}}^{\mathrm{univ}}$. Thus $$a\left(l, \mathcal{F} \right)-\chi\left(l \right)\langle l \rangle{\kappa^{\prime}}^{-1}\left(\langle l \rangle \right)-1=a\left(\mathrm{Frob}_l \right)+d\left(\mathrm{Frob}_l \right)-\chi\kappa_{\mathrm{cyc}}\kappa_{\mathrm{cyc}}^{\mathrm{univ}}\left(\mathrm{Frob}_l \right)-1$$ is an element of $I\left(\rho_{\mathcal{F}} \right)$ for all primes $l \nmid Np$. Since $\varepsilon_2\left(\mathrm{Frob}_p \right)=a\left(p, \mathcal{F} \right)$ by \cite[Theorem 2.2.2]{Wi88}, we also have $a\left(p, \mathcal{F} \right)-1 \in I\left(\rho_{\mathcal{F}} \right)$. This completes the proof.

\end{proof}

\subsection{Proof of Theorem \ref{yandongmain} its corollary}
Under the above preparation, we return to the proof of Theorem \ref{yandongmain} and Corollary \ref{canonical}. First we obtain the following lemma on the relation between the ideal of reducibility and its localization. 
\begin{lem}\label{redu}
Let us keep the assumptions and the notation of Theorem \ref{yandongmain}. For any $\mathfrak{p}\in P^{1}\left(\mathbb{I} \right)$, we denote by $I\left(\rho_{\mathcal{F}, \mathfrak{p}} \right)$ the ideal of reducibility of $\mathbb{I}_{\mathfrak{p}}$ corresponding to the representation $$\rho_{\mathcal{F}, \mathfrak{p}}: G_{\mathbb{Q}} \stackrel{\rho_{\mathcal{F}}}{\rightarrow} \mathrm{GL}_2\left(\mathbb{K} \right)=\mathrm{GL}_2\left(\mathrm{Frac}\left(\mathbb{I}_{\mathfrak{p}} \right) \right).$$ Then we have $I\left(\rho_{\mathcal{F}, \mathfrak{p}} \right)=I\left(\rho_{\mathcal{F}} \right)\mathbb{I}_{\mathfrak{p}}$.
\end{lem}
\begin{proof}
The proof is by definition. Clearly we have $\mathrm{tr}\,\rho_{\mathcal{F}} \subset \mathbb{I} \subset \mathbb{I}_{\mathfrak{p}}$. By the proof of Lemma \ref{3.7}, there exist a $\mathbb{K}$-basis $\set{e_1, e_2}$ of $\mathbb{V}_{\mathcal{F}}$ and an element $g_0 \in D_p$ such that 
\begin{equation}
\rho_{\mathcal{F}}(g_0) = \begin{pmatrix} \varepsilon_1\left(g_0 \right) & 0 \\ 0 & \varepsilon_2\left(g_0 \right) \end{pmatrix}, \varepsilon_1\left(g_0 \right) \not\equiv \varepsilon_2\left(g_0 \right) \pmod{\mathfrak{m}_{\mathbb{I}}}
\end{equation}
with respect to the basis $\set{e_1, e_2}$. This implies $\varepsilon_1\left(g_0 \right) \not\equiv \varepsilon_2\left(g_0 \right) \pmod{\mathfrak{p}}$. Thus, we complete the proof by constructions of $I\left(\rho_{\mathcal{F}, \mathfrak{p}} \right)$ and $I\left(\rho_{\mathcal{F}}\right)$ which are done in the proof of Proposition \ref{21}. 

\end{proof}

The assumption that $\mathbb{I}$ is a regular local ring enables us to take a $G_{\mathbb{Q}}$-stable $\mathbb{I}$-free lattice $\mathbb{T}$. We fix such $\mathbb{T}$ to the end of the proof of Lemma \ref{1108}. For any $\mathfrak{p}\in P^{1}\left(\mathbb{I} \right)$, we denote by $\mathscr{L}\left(\rho_{\mathcal{F}, \mathfrak{p}} \right)$ the set of isomorphic classes of $G_{\mathbb{Q}}$-stable $\mathbb{I}_{\mathfrak{p}}$-lattices of $\mathbb{V}_{\mathcal{F}}$. Since $\mathbb{I}_{\mathfrak{p}}$ is a discrete valuation ring, we have
\begin{equation}\label{190228}
\sharp\mathscr{C}\left(\rho_{\mathcal{F}, \mathfrak{p}} \right)=\sharp\mathscr{L}\left(\rho_{\mathcal{F}, \mathfrak{p}} \right)=\mathrm{ord}_{\mathfrak{p}}I\left(\rho_{\mathcal{F}, \mathfrak{p}} \right) +1
\end{equation}
by Proposition \ref{22}. Then by combining Lemma \ref{3.7} and Lemma \ref{redu}, the equality \eqref{190228} becomes to
\begin{equation}\label{190128}
\sharp\mathscr{C}\left(\rho_{\mathcal{F}, \mathfrak{p}} \right)=\sharp\mathscr{L}\left(\rho_{\mathcal{F}, \mathfrak{p}} \right)=\mathrm{ord}_{\mathfrak{p}}\mathbb{J}_{\mathfrak{p}} +1.
\end{equation}
We define $\mathcal{N}$ the subset of $P^1\left(\mathbb{I} \right)$ as follows:$$\mathcal{N}=\set{\mathfrak{p} \in P^{1}\left(\mathbb{I} \right) | \mathrm{ord}_{\mathfrak{p}}\mathbb{J}_{\mathfrak{p}} \neq 0}.$$ Since $\mathbb{I}$ is a unique factorization domain, every height one prime ideal of $\mathbb{I}$ is principal. Then $\mathcal{N}$ is finite. First we assume that $\mathcal{N}$ is non-empty. Let $\mathcal{N}=\set{\mathfrak{p}_1, \cdots, \mathfrak{p}_r}$ and let us take an element $\mathfrak{p}_i \in \mathcal{N}$. Write $\mathrm{ord}_{\mathfrak{p}_i}\mathbb{J}_{\mathfrak{p}_i}=n_i$. We have $$\mathrm{tr}\,\rho_{\mathcal{F}} \equiv \chi\kappa_{\mathrm{cyc}}\kappa_{\mathrm{cyc}}^{\mathrm{univ}}+\mathbf{1} \pmod{\mathfrak{p}_i^{n_i}\mathbb{I}_{\mathfrak{p}_i}}$$ by Lemma \ref{3.7}. Then by Proposition \ref{22}-(3), there exists a chain of $G_{\mathbb{Q}}$-stable $\mathbb{I}_{\mathfrak{p}_i}$-lattices $$\mathbb{T}_i^{(n_i)} \supsetneq \cdots \supsetneq \mathbb{T}_i^{(0)}$$ of $\mathbb{V}_{\mathcal{F}}$ such that $\mathbb{T}_i^{(n_i)} \supsetneq \cdots \supsetneq \mathbb{T}_i^{(0)}$ is a system of  representatives of $\mathscr{C}\left(\rho_{\mathcal{F}, \mathfrak{p}_i} \right)$ and $\mathscr{L}\left(\rho_{\mathcal{F}, \mathfrak{p}_i} \right)$ which satisfies the following condition:
\begin{list}{}{}
\item[(Type $\mathbf{1}$)]The $\mathbb{I}_{\mathfrak{p}_i}$-module $\mathbb{T}_i^{(j_i)}/\mathbb{T}_i^{(0)}$ is isomorphic to $\mathbb{I}_{\mathfrak{p}_i}/\mathfrak{p}_i^{j_i}$ for every $1 \leq j_i \leq n_i$ and $G_{\mathbb{Q}}$ acts on $\mathbb{T}_i^{(n_i)}/\mathbb{T}_i^{(0)}$ trivially.
\end{list}
This implies that the $\mathbb{I}_{\mathfrak{p}_i}[G_{\mathbb{Q}}]$-module $\mathbb{T}_i^{(j_i)}/\mathbb{T}_i^{(0)}$ is isomorphic to $\mathbb{I}_{\mathfrak{p}_i}/\mathfrak{p}_i^{j_i} \left(\mathbf{1} \right)$ for every $j_i=1, \cdots, n_i$.

For each $\mathfrak{p}_i \in \mathcal{N}$ and $0 \leq j_i \leq n_i$, we define the module $\mathbb{T}\left(j_1, \cdots, j_r \right)$ as follows:
\begin{equation}\label{04281}
\mathbb{T}\left(j_1, \cdots, j_r \right)=\bigcap_{\mathfrak{p} \not\in \mathcal{N}}\mathbb{T}_{\mathfrak{p}}\bigcap_{\mathfrak{p}_i \in \mathcal{N}}\mathbb{T}_i^{(j_i)}.
\end{equation}
Then $\mathbb{T}\left(j_1, \cdots, j_r \right)$ is a reflexive lattice and we have
\begin{equation}\label{1901128}
\mathbb{T}\left(j_1, \cdots, j_r \right)_{\mathfrak{p}}=\begin{cases}
\mathbb{T}_{\mathfrak{p}} & \left(\mathfrak{p} \not\in \mathcal{N} \right) \\
\mathbb{T}_i^{\left(j_i \right)} & \left(\mathfrak{p}=\mathfrak{p}_i\in \mathcal{N} \right) \\
\end{cases}
\end{equation}
by \cite[Chap. \RomanNumeralCaps{7}. \S 4.3, Theorem 3-(ii)]{Bour}. Since every $\mathbb{T}_{\mathfrak{p}}$ and $\mathbb{T}_i^{\left(j_i \right)}$ are stable under the action of $G_{\mathbb{Q}}$, $\mathbb{T}\left(j_1, \cdots, j_r \right)$ is a $G_{\mathbb{Q}}$-stable lattice. Furthermore, under the assumption that $\mathbb{I}$ is a regular local ring, $\mathbb{T}\left(j_1, \cdots, j_r \right)$ is free over $\mathbb{I}$. Thus $\mathbb{T}\left(j_1, \cdots, j_r \right)$ is a $G_{\mathbb{Q}}$-stable $\mathbb{I}$-free lattice of $\mathbb{V}_{\mathcal{F}}$. 
\begin{lem}\label{1108}
When $\mathcal{N}$ is non-empty, $\set{\mathbb{T}\left(j_1, \cdots, j_r \right) | i=1, \cdots, r, j_i=0, \cdots, n_i}$ is a set of representatives of $\mathscr{L}^{\mathrm{fr}}\left(\rho_{\mathcal{F}} \right)$ and $\mathscr{C}^{\mathrm{fr}}\left(\rho_{\mathcal{F}} \right)$. When $\mathcal{N}$ is empty, $\mathscr{L}^{\mathrm{fr}}\left(\rho_{\mathcal{F}} \right)$ (resp. $\mathscr{C}^{\mathrm{fr}}\left(\rho_{\mathcal{F}} \right)$) consist only of the isomorphic class (resp. homothetic class) of $\mathbb{T}$.
\end{lem}
\begin{proof}
First we assume that $\mathcal{N}$ is non-empty. Let us take a $G_{\mathbb{Q}}$-stable $\mathbb{I}$-free lattice $\mathbb{T}^{\prime}$. By multiplying an element of $\mathbb{I}$ if necessary, we may assume $\mathbb{T}^{\prime} \subset \mathbb{T}\left(n_1, \cdots, n_r \right)$. Let us take an element $\mathfrak{p} \in P^1\left(\mathbb{I} \right)$ and let us consider the following cases:
\begin{enumerate}
\item[(a)]When $\mathfrak{p} \not\in \mathcal{N}$, $\mathscr{L}\left(\rho_{\mathcal{F}, \mathfrak{p}} \right)$ (resp. $\mathscr{C}\left(\rho_{\mathcal{F}, \mathfrak{p}} \right)$) consists only of the isomorphic (resp. homothetic) class of $\mathbb{T}_{\mathfrak{p}}$ by \eqref{190128}. Since $\mathbb{T}\left(n_1, \cdots, n_r \right)_{\mathfrak{p}}=\mathbb{T}_{\mathfrak{p}}$ by \eqref{1901128}, under the assumption $\mathbb{T}^{\prime} \subset \mathbb{T}\left(n_1, \cdots, n_r \right)$, there exists an integer $e_{\mathfrak{p}} \in \mathbb{Z}_{\geq 0}$ such that $\mathbb{T}_{\mathfrak{p}}^{\prime}=\mathfrak{p}^{e_{\mathfrak{p}}}\mathbb{T}_{\mathfrak{p}}.$
\item[(b)]When $\mathfrak{p}=\mathfrak{p}_i \in \mathcal{N}$, $\mathbb{T}_i^{(n_i)} \supsetneq \cdots \supsetneq \mathbb{T}_i^{(0)}$ is a system of  representatives of $\mathscr{L}\left(\rho_{\mathcal{F}, \mathfrak{p}_i} \right)$ and $\mathscr{C}\left(\rho_{\mathcal{F}, \mathfrak{p}_i} \right)$. Then under the assumption $\mathbb{T}^{\prime} \subset \mathbb{T}\left(n_1, \cdots, n_r \right)$, there exist an integer $e_{\mathfrak{p}_i} \in \mathbb{Z}_{\geq 0}$ and an integer $0 \leq j_i \leq n_i$ such that $\mathbb{T}_{\mathfrak{p}_i}=\mathfrak{p}_i^{e_{\mathfrak{p}_i}}\mathbb{T}_i^{\left(j_i \right)}$. 
\end{enumerate}
We have $\mathbb{T}^{\prime}_{\mathfrak{p}}=\mathbb{T}_{\mathfrak{p}}$ for all but finitely many $\mathfrak{p} \in P^1\left(\mathbb{I} \right)$ by \cite[Chap. \RomanNumeralCaps{7}. \S 4.3, Theorem 3-(i)]{Bour}. Thus the integers $e_{\mathfrak{p}}$ in cases (a) and (b) are $0$ for all but finitely many $\mathfrak{p}\in P^1\left(\mathbb{I} \right)$. Furthermore, since $\mathbb{I}$ is a UFD, every height-one prime ideal is principal. Then $\displaystyle\prod_{\mathfrak{p} \in P^{1}\left(\mathbb{I} \right)}\mathfrak{p}^{e_{\mathfrak{p}}}$ is generated by an element $x \in \mathbb{I}$. Let $\mathbb{T}^{\prime\prime}=x\mathbb{T}\left(j_1, \cdots, j_r \right).$ We have $\mathbb{T}^{\prime}_{\mathfrak{p}}=\mathbb{T}^{\prime\prime}_{\mathfrak{p}}$ for any $\mathfrak{p} \in P^{1}\left(\mathbb{I} \right)$. Since $\mathbb{T}^{\prime}$ and $\mathbb{T}^{\prime\prime}$ are reflexive lattices, we have $$\mathbb{T}^{\prime}=\bigcap_{\mathfrak{p} \in P^{1}\left(\mathbb{I} \right)}\mathbb{T}^{\prime}_{\mathfrak{p}}=\bigcap_{\mathfrak{p} \in P^{1}\left(\mathbb{I} \right)}\mathbb{T}^{\prime\prime}_{\mathfrak{p}}=\mathbb{T}^{\prime\prime}.$$ This implies that $\set{\mathbb{T}\left(j_1, \cdots, j_r \right) | i=1, \cdots, r, j_i=0, \cdots, n_i}$ is a system of  representatives of $\mathscr{C}^{\mathrm{fr}}\left(\rho_{\mathcal{F}} \right)$.

By our construction, there exists a prime ideal $\mathfrak{p} \in P^{1}\left(\mathbb{I} \right)$ such that $\mathbb{T}\left(j_1, \cdots, j_r \right)_{\mathfrak{p}}$ and $\mathbb{T}\left(j_1^{\prime}, \cdots, j_r^{\prime} \right)_{\mathfrak{p}}$ are non-isomorphic as $\mathbb{I}_{\mathfrak{p}}[G_{\mathbb{Q}}]$-modules for $\left(j_1, \cdots, j_r \right) \neq \left(j_1^{\prime}, \cdots, j_r^{\prime} \right)$. Thus lattices $\mathbb{T}\left(j_1, \cdots, j_r \right)$ and $\mathbb{T}\left(j_1^{\prime}, \cdots, j_r^{\prime} \right)$ are not isomorphic as $\mathbb{I}[G_{\mathbb{Q}}]$-modules. This implies that $\set{\mathbb{T}\left(j_1, \cdots, j_r \right) | i=1, \cdots, r, j_i=0, \cdots, n_i}$ is also a system of  representatives of $\mathscr{L}^{\mathrm{fr}}\left(\rho_{\mathcal{F}} \right)$.

Now we assume that $\mathcal{N}$ is empty. Let us take a $G_{\mathbb{Q}}$-stable $\mathbb{I}$-free lattice $\mathbb{T}^{\prime}$. By multiplying an element of $\mathbb{I}$ if necessary, we may assume $\mathbb{T}^{\prime} \subset \mathbb{T}$. Under the assumption that $\mathcal{N}$ is empty, any prime ideal $\mathfrak{p} \in P^1\left(\mathbb{I} \right)$ belongs to case (a) above. Then by the same argument, we have that there exists an element $x^{\prime}\in\mathbb{I}$ such that $\mathbb{T}^{\prime}=x^{\prime}\mathbb{T}$. This completes the proof of Lemma \ref{1108}.

\end{proof}

When $\mathbb{J}^{**}=\mathbb{I}$, $\mathscr{L}^{\mathrm{fr}}\left(\rho_{\mathcal{F}} \right)$ consists of a unique element by Lemma \ref{1108}, hence all statements in Theorem \ref{yandongmain} follows. Thus it is sufficient to consider the case when $\mathbb{J}^{**} \subsetneq \mathfrak{m}_{\mathbb{I}}$ in the rest of the proof. This enables us to take an element $\left(j_1, \cdots, j_r \right) \neq \left(0, \cdots, 0 \right)$ which we fixed in the rest of the proof. 

\begin{dfn}
Let $A$ be a Noetherian local domain and $M$ a finitely generated $A$-module. Let $\delta: G \rightarrow A^{\times}$ be a character of a group $G$. We call that $M$ is an $A[G]$-module of \textit{type $\delta$} if $M$ is a cyclic $A$-module and $G$ acts on $M$ via $\delta$.
\end{dfn}

\begin{lem}\label{11081}
We have the equality $F^{+}\mathbb{T}\left(j_1, \cdots, j_r \right)=F^{+}\mathbb{T}\left(0, \cdots, 0 \right)$.
\end{lem}
\begin{proof}
Write $\mathbb{T}_{\mathbf{0}}=\mathbb{T}\left(0, \cdots, 0 \right)$ and $\mathbb{T}_{\mathbf{j}}=\mathbb{T}\left(j_1, \cdots, j_r \right)$ for short. We have $\mathbb{T}_{\mathbf{0}} \subset \mathbb{T}_{\mathbf{j}}$ by our construction, hence $F^{+}\mathbb{T}_{\mathbf{0}} \subset F^{+}\mathbb{T}_{\mathbf{j}}$. We prove this Lemma by contradiction. Assume $F^{+}\mathbb{T}_{\mathbf{0}} \subsetneq F^{+}\mathbb{T}_{\mathbf{j}}$. Let us consider the following commutative diagram of $\mathbb{I}[D_p]$-modules:
\begin{equation}\label{Keycom}
 \xymatrix{
   0 \ar[r]   & F^{+}\mathbb{T}_{\mathbf{0}} \ar[r] \ar@{^{(}-_>}[d] & \mathbb{T}_{\mathbf{0}} \ar[r] \ar@{^{(}-_>}[d] & F^{-}\mathbb{T}_{\mathbf{0}} \ar[r] \ar[d] & 0 \\
    0 \ar[r] & F^{+}\mathbb{T}_{\mathbf{j}} \ar[r] & \mathbb{T}_{\mathbf{j}} \ar[r] & F^{-}\mathbb{T}_{\mathbf{j}} \ar[r] & 0.
  }
\end{equation}
We have the following exact sequence by the snake lemma:
\begin{equation}\label{190108}
0 \rightarrow \mathrm{Ker}\left(F^{-}\mathbb{T}_{\mathbf{0}} \rightarrow F^{-}\mathbb{T}_{\mathbf{j}} \right) \rightarrow \left.F^{+}\mathbb{T}_{\mathbf{j}}\right/F^{+}\mathbb{T}_{\mathbf{0}} \rightarrow \left.\mathbb{T}_{\mathbf{j}}\right/\mathbb{T}_{\mathbf{0}} \rightarrow \left.F^{-}\mathbb{T}_{\mathbf{j}}\right/F^{-}\mathbb{T}_{\mathbf{0}} \rightarrow 0.
\end{equation}
Since $F^{+}\mathbb{T}_{\mathbf{0}} \subsetneq F^{+}\mathbb{T}_{\mathbf{j}}$ are free $\mathbb{I}$-modules of rank one by Lemma \ref{FO}, we have $F^{+}\mathbb{T}_{\mathbf{0}} \subset \mathfrak{m}_{\mathbb{I}}F^{+}\mathbb{T}_{\mathbf{j}}$. Then by the condition ($D_p$-dist), we must have $$\mathrm{Ker}\left(F^{-}\mathbb{T}_{\mathbf{0}} \rightarrow F^{-}\mathbb{T}_{\mathbf{j}} \right)=0.$$Then \eqref{190108} becomes to
\begin{equation}\label{1901092}
0 \rightarrow \left.F^{+}\mathbb{T}_{\mathbf{j}}\right/F^{+}\mathbb{T}_{\mathbf{0}} \rightarrow \left.\mathbb{T}_{\mathbf{j}}\right/\mathbb{T}_{\mathbf{0}} \rightarrow \left.F^{-}\mathbb{T}_{\mathbf{j}}\right/F^{-}\mathbb{T}_{\mathbf{0}} \rightarrow 0.
\end{equation}

Recall that $F^{+}\mathbb{T}_{\mathbf{j}}$ (resp. $F^{-}\mathbb{T}_{\mathbf{j}}$) is an $\mathbb{I}[D_p]$-module of type $\varepsilon_1$ (resp. $\varepsilon_2$). Let us take an element $\mathfrak{p}_i \in \mathcal{N}$ such that $\left(\mathbb{T}_{\mathbf{j}}/\mathbb{T}_{\mathbf{0}} \right)_{\mathfrak{p}_i}\cong\mathbb{T}_{\mathbf{j}, \mathfrak{p}_i}/\mathbb{T}_{\mathbf{0}, \mathfrak{p}_i}\cong\mathbb{I}_{\mathfrak{p}_i}/\mathfrak{p}_i^{j_i}$ with $j_i \neq 0$. By localizing the exact sequence \eqref{1901092} at $\mathfrak{p}_i$, we have that $\left(F^{+}\mathbb{T}_{\mathbf{j}}/F^{+}\mathbb{T}_{\mathbf{0}} \right)_{\mathfrak{p}_i}$ is a type $\varepsilon_1$ $\mathbb{I}_{\mathfrak{p}_i}[D_p]$-submodule of $\left(\mathbb{T}_{\mathbf{j}}/\mathbb{T}_{\mathbf{0}} \right)_{\mathfrak{p}_i}$. However by the condition (Type $\mathbf{1}$), $\mathbb{T}_{\mathbf{j}, \mathfrak{p}_i}/\mathbb{T}_{\mathbf{0}, \mathfrak{p}_i}$ is an $\mathbb{I}_{\mathfrak{p}_i}[G_{\mathbb{Q}}]$-module of type $\mathbf{1}$. This contradicts to the condition ($D_p$-dist).

\end{proof}

\begin{lem}\label{Fil2}
We have the following isomorphism of $\mathbb{I}[G_{\mathbb{Q}}]$-modules: $$\left.\mathbb{T}\left(j_1, \cdots, j_r \right)\right/{\mathbb{T}\left(0, \cdots, 0 \right)}\stackrel{\sim}{\rightarrow}\bigslant{\mathbb{I}}{\displaystyle\prod_{i=1}^{r}\mathfrak{p}_i^{j_i}}\ \left(\mathbf{1} \right).$$
\end{lem}
\begin{proof}
Let us keep the notation which is used in the proof of Lemma \ref{11081}. The commutative diagram \eqref{Keycom} induces the following isomorphism of $\mathbb{I}[D_p]$-modules
\begin{equation}\label{1901087}
\mathbb{T}_{\mathbf{j}}/\mathbb{T}_{\mathbf{0}} \stackrel{\sim}{\rightarrow} F^{-}\mathbb{T}_{\mathbf{j}}/F^{-}\mathbb{T}_{\mathbf{0}}
\end{equation}
by Lemma \ref{11081} and its proof. Since the $\mathbb{I}$-modules $F^{-}\mathbb{T}_{\mathbf{0}}$ and $F^{-}\mathbb{T}_{\mathbf{j}}$ are free of rank one by Lemma \ref{FO}, under our assumption $\left(j_1, \cdots, j_r \right) \neq \left(0, \cdots, 0 \right)$, there exists an element $\xi \in \mathfrak{m}_{\mathbb{I}}$ such that $\mathbb{T}_{\mathbf{j}}/\mathbb{T}_{\mathbf{0}} \stackrel{\sim}{\rightarrow} \mathbb{I}/\left(\xi \right)$ as $\mathbb{I}$-modules. By our construction, we have $\mathbb{T}_{\mathbf{0}, \mathfrak{p}}=\mathbb{T}_{\mathbf{j}, \mathfrak{p}}$ for any $\mathfrak{p} \not\in \mathcal{N}$ and $\mathbb{T}_{\mathbf{j}, \mathfrak{p}_i}/\mathbb{T}_{\mathbf{0}, \mathfrak{p}_i}\cong\mathbb{I}_{\mathfrak{p}_i}/\mathfrak{p}_i^{j_i}$ for any $\mathfrak{p}_i \in \mathcal{N}$. Thus we have
\begin{equation}\label{1901086}
\left(\xi \right)=\mathrm{char}_{\mathbb{I}}\left(\mathbb{T}_{\mathbf{j}}/\mathbb{T}_{\mathbf{0}} \right)=\displaystyle\prod_{\mathfrak{p} \in P^1\left(\mathbb{I} \right)}\mathfrak{p}^{\mathrm{length}_{\mathbb{I}_{\mathfrak{p}}}\left(\mathbb{T}_{\mathbf{j}}/\mathbb{T}_{\mathbf{0}} \right)_{\mathfrak{p}}}=\displaystyle\prod_{i=1}^{r}\mathfrak{p}_i^{j_i}.
\end{equation}

By the following isomorphism of $\mathbb{I}$-modules $$\mathbb{T}_{\mathbf{j}}/\mathbb{T}_{\mathbf{0}} \stackrel{\sim}{\rightarrow} \left.F^{-}\mathbb{T}_{\mathbf{j}} \right/F^{-}\mathbb{T}_{\mathbf{0}} \stackrel{\sim}{\rightarrow} \mathbb{I}/\left(\xi \right),$$ we have $\xi F^{-}\mathbb{T}_{\mathbf{j}}=F^{-}\mathbb{T}_{\mathbf{0}}$ and $\xi \mathbb{T}_{\mathbf{j}} \subset \mathbb{T}_{\mathbf{0}}.$ Thus by the following commutative diagram
\begin{equation*}
 \xymatrix{
    0 \ar[r] & \xi F^{+}\mathbb{T}_{\mathbf{j}} \ar[r] \ar@{^{(}-_>}[d] & \xi\mathbb{T}_{\mathbf{j}} \ar[r] \ar@{^{(}-_>}[d] & \xi F^{-}\mathbb{T}_{\mathbf{j}} \ar[r] \ar@{=}[d] & 0 \\
    0 \ar[r] & F^{+}\mathbb{T}_{\mathbf{0}} \ar[r] & \mathbb{T}_{\mathbf{0}} \ar[r] & F^{-}\mathbb{T}_{\mathbf{0}} \ar[r] & 0,
  }
\end{equation*}
we have that $\left.\mathbb{T}_{\mathbf{0}}\right/\xi\mathbb{T}_{\mathbf{j}}$ is isomorphic to $\left.F^{+}\mathbb{T}_{\mathbf{0}}\right/\xi F^{+}\mathbb{T}_{\mathbf{j}}$. Thus $\left.\mathbb{T}_{\mathbf{0}}\right/\xi\mathbb{T}_{\mathbf{j}}$ is isomorphic to $\left.\mathbb{I}\right/{\left(\xi\right)}$ as an $\mathbb{I}$-module by Lemma \ref{11081}. Then the following exact sequence $$0 \rightarrow \mathbb{T}_{\mathbf{0}}/\xi\mathbb{T}_{\mathbf{j}} \rightarrow \mathbb{T}_{\mathbf{j}}/\xi\mathbb{T}_{\mathbf{j}} \rightarrow \mathbb{T}_{\mathbf{j}}/\mathbb{T}_{\mathbf{0}} \rightarrow 0$$ implies that $\mathrm{tr}_{\mathcal{F}}\,\mathrm{mod}\,\left(\xi \right)$ is the sum of two characters with values in $\left(\mathbb{I}/\left(\xi \right) \right)^{\times}$. Then we have that $\mathbb{T}_{\mathbf{j}}/\mathbb{T}_{\mathbf{0}}$ is an $\mathbb{I}[G_{\mathbb{Q}}]$-module of either type $\chi\kappa_{\mathrm{cyc}}\kappa_{\mathrm{cyc}}^{\mathrm{univ}}$ or type $\mathbf{1}$ by combining Proposition \ref{21} with Lemma \ref{3.7}. Since $\mathbb{T}_{\mathbf{j}}/\mathbb{T}_{\mathbf{0}}$ is a type $\varepsilon_2$ $\mathbb{I}[D_p]$-module by \eqref{1901087}, we must have 
\begin{equation}\label{190109}
\mathbb{T}_{\mathbf{j}}/\mathbb{T}_{\mathbf{0}} \stackrel{\sim}{\rightarrow} \mathbb{I}/\left(\xi \right)\,\left(\mathbf{1} \right)
\end{equation}
under ($D_p$-dist). Thus, we complete the proof by combining \eqref{190109} with \eqref{1901086}.
\end{proof}

Let us return to the proof of Theorem \ref{yandongmain}.
\begin{proof}[Proof of Theorem \ref{yandongmain}]
When $\mathbb{J}^{**}=\mathbb{I}$, $\mathscr{L}^{\mathrm{fr}}\left(\rho_{\mathcal{F}} \right)$ consists of a unique element by Lemma \ref{1108}, hence all statements in Theorem \ref{yandongmain} follows. Thus we may assume $\mathbb{J}^{**} \subsetneq \mathbb{I}$. Let $$\mathscr{T}=\set{\mathbb{T}\left(j_1, \cdots, j_r \right) | i=1, \cdots, r, j_i=0, \cdots, n_i},$$ where $\mathbb{T}\left(j_1, \cdots, j_r \right)$ is defined as \eqref{04281}. Then by Lemma \ref{1108}, we have the following equality
\begin{equation*}
\sharp\mathscr{C}^{\mathrm{fr}}\left(\rho_{\mathcal{F}} \right)=\sharp\mathscr{L}^{\mathrm{fr}}\left(\rho_{\mathcal{F}} \right)=\displaystyle\prod_{\mathfrak{p} \in P^1\left(\mathbb{I} \right)}\left(\mathrm{ord}_{\mathfrak{p}}\mathbb{J}_{\mathfrak{p}} +1 \right).
\end{equation*}
We have $\left(\mathcal{L}_p\left(\chi; \gamma^{\prime} \right) \right) \subset \mathbb{J}$ by \cite[Proposition 3.8]{Y}, thus the first assertion of Theorem \ref{yandongmain} follows. Let 
\begin{equation}\label{12021}
\mathbb{T}^{\mathrm{min}}=\mathbb{T}\left(0, \cdots, 0\right).
\end{equation}
Then the second assertion of Theorem \ref{yandongmain} follows by Lemma \ref{Fil2}. 

Now we prove the third assertion. Let $\mathbb{T}$ be a $G_{\mathbb{Q}}$-stable lattice of $\mathbb{V}_{\mathcal{F}}$ and $\mathcal{T}^{\left(i \right)}=\mathbb{T}\hat{\otimes}_{\mathbb{Z}_p}\mathbb{Z}_p[[\Gamma]]\left(\tilde{\kappa}^{-1} \right)\otimes\omega^i$. We compute the quotient $\dfrac{(L_p^{\mathrm{alg}}(  \mathcal{T}^{(i )} )    )}{(L_p^{\mathrm{alg}}(  \mathcal{T}^{\mathrm{min}, (i )} )    )}$. By Lemma \ref{43}, it is enough to assume that $\mathbb{T}$ is a free lattice. Furthermore by the second assertion of Theorem \ref{yandongmain}, it is enough to assume $\mathbb{T} \in \mathscr{T}$. Then we have $F^{+}\mathcal{T}^{\left(i \right)}=F^{+}\mathcal{T}^{\mathrm{min}, \left(i \right)}$ and there exists a factor $\mathfrak{A}$ of $\mathcal{J}^{**}$ such that 
\begin{equation}\label{10251}
\left.\mathcal{T}^{\left(i \right)}\right/\mathcal{T}^{\mathrm{min}, \left(i \right)} \stackrel{\sim}{\rightarrow} \left.\mathcal{R}\right/\mathfrak{A}\,\left(\omega^i\tilde{\kappa}^{-1} \right)
\end{equation}
by the second assertion of Theorem \ref{yandongmain}. Since $\mathbb{Q}_{\infty}$ is totally real, $\left.\tilde{\kappa}\right|_{G_{\mathbb{R}}}$ is a trivial character. Thus by applying Ochiai's theorem (Theorem \ref{O}) to the $\mathcal{T}^{\left(i \right)}$ and $\mathcal{T}^{\mathrm{min}, \left(i \right)}$, we have 
\begin{equation*}
\dfrac{(L_p^{\mathrm{alg}}(  \mathcal{T}^{(i )} )    )}{(L_p^{\mathrm{alg}}(  \mathcal{T}^{\mathrm{min}, (i )} )    )}=\begin{cases}
1& \left(i: \text{odd} \right) \\
\mathfrak{A}& \left(i: \text{even} \right).
\end{cases}
\end{equation*}
Since the correspondence between $\left.\mathfrak{A}\right|\mathcal{J}^{**}$ and $\mathbb{T} \in \mathscr{T}$ is one-to-one by the second assertion of Theorem \ref{yandongmain}, this completes the proof of Theorem \ref{yandongmain}.

\end{proof}

By Theorem \ref{yandongmain}-(2), there exists an $\mathbb{I}$-free lattice $\mathbb{T}^{\mathrm{max}} \in \mathscr{T}$ such that 
\begin{equation}\label{12011}
\left.\mathbb{T}^{\mathrm{max}}\right/{\mathbb{T}^{\mathrm{min}}} \stackrel{\sim}{\rightarrow} \left.\mathbb{I}\right/{\mathbb{J}^{**}}\,\left(\mathbf{1}\right).
\end{equation}
By the proof of Theorem \ref{yandongmain}, we have that the algebraic $p$-adic $L$-function $L_p^{\mathrm{alg}}\left(\mathcal{T}^{\mathrm{max}, \left(i \right)} \right)$ for the lattice $\mathbb{T}^{\mathrm{max}}$ is maximal under divisibility among the set of $L_p^{\mathrm{alg}}\left(\mathcal{T}^{\left(i\right)} \right)$ for all $G_{\mathbb{Q}}$-stable lattices of $\mathbb{V}_{\mathcal{F}}$. Now we give a geometric characterization of the lattices $\mathbb{T}^{\mathrm{min}}$ and $\mathbb{T}^{\mathrm{max}}$. Let $\mathbf{H}^{\mathrm{ord}}:=\mathbf{H}^{\mathrm{ord}}\left(N, \mathbb{Z}_p[\chi]\right)$ be Hida's ordinary Hecke algebra over $\Lambda_{\chi}$ as in \cite{H1} and $\mathbf{h}^{\mathrm{ord}}:=\mathbf{h}^{\mathrm{ord}}\left(N, \mathbb{Z}_p[\chi]\right)$ the quotient of $\mathbf{H}^{\mathrm{ord}}$ corresponding to cusp forms. Let $\mathfrak{M}:=\mathfrak{M}\left(\chi\omega^{-1}, \mathbf{1} \right)$ be the Eisenstein maximal ideal of $\mathbf{H}^{\mathrm{ord}}$ as in \cite[\S 1.2, (1.2.9)]{MO05}. Note that $\mathbb{I}$ is an extension of a quotient of $\mathbf{h}^{\mathrm{ord}}_{\mathfrak{M}}$. By combining Theorem \ref{yandongmain} with a result of Ohta \cite[\S 3.4]{MO05}, we have the following corollary.
\begin{cor}\label{canonical}
Let us keep the assumption of Theorem \ref{yandongmain}. Assume further that rings $\mathbf{H}^{\mathrm{ord}}_{\mathfrak{M}}$ and $\mathbf{h}^{\mathrm{ord}}_{\mathfrak{M}}$ are Gorenstein. Then we have the following equalities
\begin{equation*}
\begin{cases}
\mathbb{T}^{\mathrm{min}}=\mathrm{Hom}_{\mathbb{I}}\left(\plim[r\geq1]H_{\text{\'et}}^1\left(X_1\left(Np^r\right)\otimes_{\mathbb{Q}}\overline{\mathbb{Q}}, \mathbb{Z}_p[\chi]    \right)_{\mathfrak{M}}^{\mathrm{ord}}\otimes_{\mathbf{H}_{\mathfrak{M}}^{\mathrm{ord}}}\mathbb{I}, \mathbb{I}\right) \\
\mathbb{T}^{\mathrm{max}}=\mathrm{Hom}_{\mathbb{I}}\left(\plim[r\geq1]H_{\text{\'et}}^1\left(Y_1\left(Np^r\right)\otimes_{\mathbb{Q}}\overline{\mathbb{Q}}, \mathbb{Z}_p[\chi]    \right)_{\mathfrak{M}}^{\mathrm{ord}}\otimes_{\mathbf{H}_{\mathfrak{M}}^{\mathrm{ord}}}\mathbb{I}, \mathbb{I}\right).
\end{cases}
\end{equation*}
up to homothety.
\end{cor}

\begin{proof}
Since the proof of the assertions on $\mathbb{T}^{\mathrm{min}}$ and $\mathbb{T}^{\mathrm{max}}$ are done in the same way, we only prove the assertion on $\mathbb{T}^{\mathrm{min}}$. Let $$\mathbb{T}_{\mathcal{F}}:=\mathrm{Hom}_{\mathbb{I}}\left(\plim[r\geq1]H_{\text{\'et}}^1\left(X_1\left(Np^r\right)\otimes_{\mathbb{Q}}\overline{\mathbb{Q}}, \mathbb{Z}_p[\chi]    \right)_{\mathfrak{M}}^{\mathrm{ord}}\otimes_{\mathbf{H}_{\mathfrak{M}}^{\mathrm{ord}}}\mathbb{I}, \mathbb{I}\right).$$ By Ohta \cite[(3.2.5) and (3.2.6)]{MO00}, we have that $\mathbb{T}_{\mathcal{F}}$ is a $G_{\mathbb{Q}}$-stable lattice of $\mathbb{V}_{\mathcal{F}}$ (note that the normalization of the above $\mathbb{T}_{\mathcal{F}}$ is dual to that of Ohta's paper \cite{MO00}). We have that $\mathbb{T}_{\mathcal{F}}$ is a dual lattice of $\plim[r\geq1]H_{\text{\'et}}^1\left(X_1\left(Np^r\right)\otimes_{\mathbb{Q}}\overline{\mathbb{Q}}, \mathbb{Z}_p[\chi]    \right)_{\mathfrak{M}}^{\mathrm{ord}}\otimes_{\mathbf{H}_{\mathfrak{M}}^{\mathrm{ord}}}\mathbb{I}\otimes_{\mathbb{I}}\mathbb{K}$. Hence $\mathbb{T}_{\mathcal{F}}$ is a reflexive. Thus, under the assumption that $\mathbb{I}$ is regular, $\mathbb{T}_{\mathcal{F}}$ is free over $\mathbb{I}$.

Let $I:=I\left(\chi\omega^{-1}, \mathbf{1} \right)$ denotes the Eisenstein ideal of $\mathbf{H}^{\mathrm{ord}}_{\mathfrak{M}}$ (cf. \cite[\S 1.2, (1.2.9)]{MO05}). Then by \cite[Corollary 4.1.12]{MO06}, $\mathbb{J}$ is the image of $I$ under the homomorphism $\mathbf{H}^{\mathrm{ord}}_{\mathfrak{M}} \rightarrow \mathbb{I}$ by the duality between $\mathbb{I}$-adic forms and their Hecke algebras (cf. \cite[Corollary 1.5.4]{MO05}). Thus under the assumption that $\mathbf{H}^{\mathrm{ord}}_{\mathfrak{M}}$ and $\mathbf{h}^{\mathrm{ord}}_{\mathfrak{M}}$ are Gorenstein, we have that $\mathbb{J}$ is principal by \cite[Theorem 3.3.8]{MO05}. 

In order to prove Corollary \ref{canonical} by contradiction, we assume that $\mathbb{T}_{\mathcal{F}}$ and $\mathbb{T}^{\mathrm{min}}$ are not homothetic. Then by Theorem \ref{yandongmain}-(2), there exist an element $x \in \mathbb{K}^{\times}$ and a proper ideal $\left.\mathfrak{a}\right|\mathbb{J}$ such that $x^{-1}\mathbb{T}_{\mathcal{F}}\in \mathscr{T}$ and 
\begin{equation}\label{12251}
\left.x^{-1}\mathbb{T}_{\mathcal{F}}\right/{\mathbb{T}^{\mathrm{min}}} \stackrel{\sim}{\rightarrow} \left.\mathbb{I}\right/{\mathfrak{a}}\,\left(\mathbf{1} \right).
\end{equation}
Hence $\left.\mathbb{T}_{\mathcal{F}}\right/{x\mathbb{T}^{\mathrm{min}}} \stackrel{\sim}{\rightarrow} \left.\mathbb{I}\right/{\mathfrak{a}}\,\left(\mathbf{1} \right).$ Let us consider the following commutative diagram
\begin{equation*}
 \xymatrix{
   0 \ar[r]   & F^{+}\left(x\mathbb{T}^{\mathrm{min}}\right) \ar[r] \ar@{^{(}-_>}[d] & x\mathbb{T}^{\mathrm{min}} \ar[r] \ar@{^{(}-_>}[d] & F^{-}\left(x\mathbb{T}^{\mathrm{min}}\right) \ar[r] \ar[d] & 0 \\
    0 \ar[r] & F^{+}\mathbb{T}_{\mathcal{F}} \ar[r] & \mathbb{T}_{\mathcal{F}} \ar[r] & F^{-}\mathbb{T}_{\mathcal{F}} \ar[r] & 0.
  }
\end{equation*}
By combining the assumption ($D_p$-dst) and the snake lemma, we have
\begin{equation*}
\mathrm{Ker}\left(F^{-}\left(x\mathbb{T}^{\mathrm{min}}\right) \rightarrow F^{-}\mathbb{T}_{\mathcal{F}}   \right)=0
\end{equation*}
and the following exact sequence
\begin{equation}
0 \rightarrow \left.F^{+}\mathbb{T}_{\mathcal{F}}\right/{F^{+}\left(x\mathbb{T}^{\mathrm{min}}\right)} \rightarrow \left.\mathbb{T}_{\mathcal{F}}\right/{x\mathbb{T}^{\mathrm{min}}} \rightarrow \left.F^{-}\mathbb{T}_{\mathcal{F}}\right/{F^{-}\left(x\mathbb{T}^{\mathrm{min}}\right)} \rightarrow 0.
\end{equation}
Then by combining ($D_p$-dist) and \eqref{12251}, we have $F^{+}\mathbb{T}_{\mathcal{F}}=F^{+}\left(x\mathbb{T}^{\mathrm{min}}\right)$ and 
\begin{equation}\label{12252}
\left.\mathbb{T}_{\mathcal{F}}\right/{x\mathbb{T}^{\mathrm{min}}} \stackrel{\sim}{\rightarrow} \left.F^{-}\mathbb{T}_{\mathcal{F}}\right/{F^{-}\left(x\mathbb{T}^{\mathrm{min}}\right)}\stackrel{\sim}{\rightarrow} \left.\mathbb{I}\right/{\mathfrak{a}}\,\left(\mathbf{1} \right).
\end{equation}

Since $\mathbf{H}^{\mathrm{ord}}_{\mathfrak{M}}$ and $\mathbf{h}^{\mathrm{ord}}_{\mathfrak{M}}$ are Gorenstein, by the proof of \cite[Corollary 3.4.13]{MO05}, we have the following exact sequence 
\begin{equation*}
0 \rightarrow \left.\mathbb{I}\right/\mathbb{J}\,\left(\mathbf{1} \right) \rightarrow \left.\mathbb{T}_{\mathcal{F}}\right/\mathbb{J}\mathbb{T}_{\mathcal{F}} \rightarrow \left.\mathbb{I}\right/\mathbb{J}\,\left(\chi\kappa_{\mathrm{cyc}}\kappa_{\mathrm{cyc}}^{\mathrm{univ}} \right) \rightarrow 0.
\end{equation*}
Let $$\mathbb{T}_{\mathcal{F}}^{\prime}:=\mathrm{Ker}\left( \mathbb{T}_{\mathcal{F}} \twoheadrightarrow \left.\mathbb{T}_{\mathcal{F}}\right/\mathbb{J}\mathbb{T}_{\mathcal{F}} \twoheadrightarrow \left.\mathbb{I}\right/\mathbb{J}\,\left(\chi\kappa_{\mathrm{cyc}}\kappa_{\mathrm{cyc}}^{\mathrm{univ}} \right)   \right).$$Then we have 
\begin{equation}\label{2001021}
\begin{cases}
F^{+}\mathbb{T}_{\mathcal{F}}^{\prime}=\mathbb{J}F^{+}\mathbb{T}_{\mathcal{F}} \\
F^{-}\mathbb{T}_{\mathcal{F}}^{\prime}=F^{-}\mathbb{T}_{\mathcal{F}}.
\end{cases}
\end{equation}
by the same arguments as the proof of Lemma \ref{11081} and Lemma \ref{Fil2}. Since $F^{+}\mathbb{T}_{\mathcal{F}}$, $F^{-}\mathbb{T}_{\mathcal{F}}$ and $\mathbb{J}$ are free $\mathbb{I}$-modules of rank one, so are the $\mathbb{I}$-modules $F^{+}\mathbb{T}_{\mathcal{F}}^{\prime}$ and $F^{-}\mathbb{T}_{\mathcal{F}}^{\prime}$. Thus $\mathbb{T}_{\mathcal{F}}^{\prime}$ is an $\mathbb{I}$-free lattice. By \eqref{2001021}, we have $F^{+}\left(\mathbb{J}^{-1}\mathbb{T}_{\mathcal{F}}^{\prime} \right)=F^{+}\mathbb{T}_{\mathcal{F}}$ and 
\begin{equation}\label{12253}
\left.\mathbb{J}^{-1}\mathbb{T}_{\mathcal{F}}^{\prime}\right/{\mathbb{T}_{\mathcal{F}}} \stackrel{\sim}{\rightarrow} \left.F^{-}\left(\mathbb{J}^{-1}\mathbb{T}_{\mathcal{F}}^{\prime}\right)\right/{F^{-}\mathbb{T}_{\mathcal{F}}} \stackrel{\sim}{\rightarrow} \left.\mathbb{I}\right/\mathbb{J}\,\left(\mathbf{1} \right).
\end{equation}
Thus $x\mathbb{T}^{\mathrm{min}} \subset \mathbb{J}^{-1}\mathbb{T}_{\mathcal{F}}^{\prime}$ and $F^{+}\left( x\mathbb{T}^{\mathrm{min}}   \right)=F^{+}\left( \mathbb{J}^{-1}\mathbb{T}_{\mathcal{F}}^{\prime} \right)$. Let us consider the following commutative diagram
\begin{equation*}
 \xymatrix{
   0 \ar[r]   & F^{+}\left(x\mathbb{T}^{\mathrm{min}}\right) \ar[r] \ar@{=}[d] & x\mathbb{T}^{\mathrm{min}} \ar[r] \ar@{^{(}-_>}[d] & F^{-}\left(x\mathbb{T}^{\mathrm{min}}\right)  \ar[r] \ar[d] & 0 \\
    0 \ar[r] & F^{+}\left(\mathbb{J}^{-1}\mathbb{T}_{\mathcal{F}}^{\prime}\right) \ar[r] & \mathbb{J}^{-1}\mathbb{T}_{\mathcal{F}}^{\prime} \ar[r] & F^{-}\left(\mathbb{J}^{-1}\mathbb{T}_{\mathcal{F}}^{\prime}\right) \ar[r] & 0.
  }
\end{equation*}
By combining \eqref{12252} and \eqref{12253}, we have 
\begin{equation*}
\left.\mathbb{J}^{-1}\mathbb{T}_{\mathcal{F}}^{\prime}\right/{x\mathbb{T}^{\mathrm{min}}} \stackrel{\sim}{\rightarrow} \left.F^{-}\left(\mathbb{J}^{-1}\mathbb{T}_{\mathcal{F}}^{\prime}\right)\right/{F^{-}\left(x\mathbb{T}^{\mathrm{min}}\right)} \stackrel{\sim}{\rightarrow} \left.\mathbb{I}\right/\mathfrak{a}\mathbb{J}\,\left(\mathbf{1} \right).
\end{equation*}
For any ideal $\mathfrak{b}$ of $\mathbb{I}$ dividing $\mathfrak{a}\mathbb{J}$, let $$\mathbb{T}\left(\mathfrak{b}\right):=\mathrm{Ker}\left( \mathbb{J}^{-1}\mathbb{T}_{\mathcal{F}}^{\prime} \twoheadrightarrow \left.\mathbb{J}^{-1}\mathbb{T}_{\mathcal{F}}^{\prime}\right/{x\mathbb{T}^{\mathrm{min}}} \stackrel{\sim}{\rightarrow} \left.\mathbb{I}\right/\mathfrak{a}\mathbb{J}\,\left(\mathbf{1} \right) \twoheadrightarrow \left.\mathbb{I}\right/{\mathfrak{b}}\,\left(\mathbf{1}\right)   \right).$$
Then $\mathbb{T}\left(\mathfrak{b}\right)$ is a $G_{\mathbb{Q}}$-stable $\mathbb{I}$-free lattice. This implies that the number of $G_{\mathbb{Q}}$-stable $\mathbb{I}$-free lattices up to homothety is greater than or equal to $\displaystyle\prod_{\mathfrak{p} \in P^1\left(\mathbb{I} \right)}\left(\mathrm{ord}_{\mathfrak{p}}\mathfrak{a}\mathbb{J} +1 \right)$. Since $\mathfrak{a} \subsetneq \mathbb{I}$, this contradicts to Theorem \ref{yandongmain}-(1). Thus, $\mathbb{T}_{\mathcal{F}}$ and $\mathbb{T}^{\mathrm{min}}$ must be homothetic and we complete the proof.
\end{proof}

Now let us assume ($\Lambda$). Note that under the assumption ($\Lambda$), the ideal $\mathbb{J}$ is principal and is generated by the Kubota-Leopoldt $p$-adic $L$ function $\mathcal{L}_p\left(\chi; \gamma^{\prime} \right)$ (see \cite[Corollary 3.9]{Y} for example). Thus under the assumption ($\Lambda$), we have that the ideal $I\left(\omega^{k-2}, \mathbf{1} \right)$ is principal. This implies that $\mathbf{H}^{\mathrm{ord}}_{\mathfrak{M}}$ and $\mathbf{h}^{\mathrm{ord}}_{\mathfrak{M}}$ are Gorenstein (cf. \cite[Lemma 3.24]{BP}). Thus Theorem A follows by combining Theorem \ref{yandongmain} and Corollary \ref{canonical}.

\section{A calculation of the two-variable algebraic $p$-adic $L$-functions}

Let us keep the assumptions of Theorem \ref{yandongmain} from now on to the end of this paper. In this section, we calculate the set $\mathscr{L}_p^{\mathrm{alg}}\left(\rho_{\mathcal{F}}^{\mathrm{n. ord}, \left(i \right)}\right)$ of all algebraic $p$-adic $L$-functions when $i \equiv 0 \pmod{p-1}$ under certain conditions. By Theorem \ref{yandongmain}-(3), it is enough to calculate $L_p^{\mathrm{alg}}\left(\mathcal{T}^{\mathrm{min}, \left(0 \right)} \right)$ for the lattice $\mathbb{T}^{\mathrm{min}}$ in Theorem \ref{yandongmain}-(2).
\begin{thm}\label{910115}
Let us keep the assumptions of Theorem \ref{yandongmain}. Assume further the following conditions
\begin{list}{}{}
\item[(Tame)]The tame level $N=1$ or $N$ is square-free.
\item[(Prim)]The character $\chi$ is primitive and $\left.\chi\right|_{\left(\left.\mathbb{Z}\right/p\mathbb{Z} \right)^{\times}}\neq \omega$.
\item[($p$Four)]The ideal $\mathbb{J}$ defined in Theorem \ref{yandongmain} is principal and is generated by $a\left(p, \mathcal{F} \right)-1$.
\end{list}
Then $L_p^{\mathrm{alg}}\left(\mathcal{T}^{\mathrm{min}, \left(0 \right)} \right)$ is a unit in $\mathcal{R}$ and $\mathscr{L}_p^{\mathrm{alg}}\left(\rho_{\mathcal{F}}^{\mathrm{n. ord}, \left(0 \right)}\right)=D\left(a\left(p, \mathcal{F} \right)-1 \right)$. 
\end{thm}

We prove Theorem \ref{910115} by specialization. The specialization method is developed by Ochiai \cite[\S 7]{Ochiai06} for the residually irreducible case. We consider the residually reducible case. Recall that $\mathbb{T}^{\mathrm{max}}$ is the $G_{\mathbb{Q}}$-stable $\mathbb{I}$-free lattice which satisfies \eqref{12011}.

\begin{lem}\label{controllattice}
Let us keep the assumptions and the notation of Theorem \ref{yandongmain}. Let $\phi$ be an arithmetic specialization. Let $T_{\phi}^{\mathrm{max}}=\mathbb{T}^{\mathrm{max}}\otimes_{\mathbb{I}}\phi\left(\mathbb{I} \right)$ and $\varpi_{\phi}$ a fixed uniformizer of $\phi\left(\mathbb{I} \right)$. Assume further that the ideal $\mathbb{J}$ is principal. Then we have the following non-split exact sequence of $\phi\left(\mathbb{I} \right)[G_{\mathbb{Q}}]$-modules $$0 \rightarrow \left.\phi\left(\mathbb{I} \right)\right/\left(\varpi_{\phi}\right)\ \left(\chi \right) \rightarrow T_{\phi}^{\mathrm{max}}/\varpi_{\phi} T_{\phi}^{\mathrm{max}} \rightarrow \left.\phi\left(\mathbb{I} \right)\right/\left(\varpi_{\phi}\right)\ \left(\mathbf{1} \right) \rightarrow 0.$$
\end{lem}
\begin{proof}
First we show that $\phi\left(\mathbb{J} \right)\neq 0$. Since $\mathbb{J}=I\left(\rho_{\mathcal{F}} \right)$ by Lemma \ref{redu}, we have $\phi\left(\mathbb{J} \right)=I\left(\rho_{f_{\phi}} \right)$ by \cite[Lemma 3.7]{Y}, where $I\left(\rho_{f_{\phi}} \right)$ is the ideal of reducibility of $\phi\left(\mathbb{I} \right)$ corresponding to $\rho_{f_{\phi}}$. Thus we have the following equality:
\begin{equation}\label{1903101}
\sharp\mathscr{L}\left(\rho_{f_{\phi}} \right)=\mathrm{ord}_{\varpi_{\phi}}\phi\left(\mathbb{J} \right)+1
\end{equation}
by \cite[Proposition 3.4]{Y}. Since $\rho_{f_{\phi}}$ is irreducible, $\sharp\mathscr{L}\left(\rho_{f_{\phi}} \right)$ must be finite. Hence $\phi\left(\mathbb{J} \right)\neq 0$ by the equality \eqref{1903101}.

Since $\mathbb{J}$ is principal, we have 
\begin{equation}\label{190412}
\left.\mathbb{T}^{\mathrm{max}}\right/\mathbb{T}^{\mathrm{min}} \stackrel{\sim}{\rightarrow} \left.\mathbb{I}\right/\mathbb{J}\,\left(\mathbf{1} \right)
\end{equation}
by \eqref{12011}. We denote by $T_{\phi}^{\mathrm{min}}$ the image of $\mathbb{T}^{\mathrm{min}}$ under $\mathbb{T}^{\mathrm{max}} \twoheadrightarrow T_{\phi}^{\mathrm{max}}$. Since $\phi\left(\mathbb{J} \right)\neq 0$, we have that $T_{\phi}^{\mathrm{min}}$ is a $G_{\mathbb{Q}}$-stable lattice of $\rho_{f_{\phi}}$ by \eqref{190412}. We have the following isomorphism:
\begin{equation}\label{1903102}
\left.T_{\phi}^{\mathrm{max}}\right/T_{\phi}^{\mathrm{min}} \stackrel{\sim}{\rightarrow} \left.\phi\left(\mathbb{I} \right)\right/\phi\left(\mathbb{J} \right)\,\left(\mathbf{1} \right).
\end{equation}
Let $l_{\phi}=\mathrm{ord}_{\varpi_{\phi}}\phi\left(\mathbb{J} \right)$ by the equality \eqref{1903101}. For any $0 \leq j \leq l_{\phi}$, let
\begin{equation*}
T_{\phi, j}=\mathrm{Ker}\left(T_{\phi}^{\mathrm{max}} \twoheadrightarrow \left.T_{\phi}^{\mathrm{max}}\right/T_{\phi}^{\mathrm{min}} \twoheadrightarrow \left.\phi\left(\mathbb{I} \right)\right/\left(\varpi_{\phi}\right)^j\,\left(\mathbf{1} \right) \right).
\end{equation*}
We have $T_{\phi, j}$ and $T_{\phi, j^{\prime}}$ are not isomorphic if $j \neq j^{\prime}$ by the proof of \cite[Proposition 3.4]{Y}. Then by \eqref{1903101}, the following chain of lattices:
\begin{equation*}\label{190310}
T_{\phi}^{\mathrm{max}}=T_{\phi, 0} \supsetneq \cdots \supsetneq T_{\phi, l_{\phi}}=T_{\phi}^{\mathrm{min}}
\end{equation*}
is a system of representatives of $\mathscr{L}\left(\rho_{f_{\phi}} \right)$. We have the following isomorphism of $\phi\left(\mathbb{I} \right)[G_{\mathbb{Q}}]$-modules for any $j=1, \cdots, l_{\phi}-1$:$$\left.T_{\phi, j}\right/\varpi_{\phi}T_{\phi, j}\cong\left.T_{\phi, j+1}\right/\varpi_{\phi}T_{\phi, j}\,\displaystyle\bigoplus\,\left.\varpi_{\phi}T_{\phi, j-1}\right/\varpi_{\phi}T_{\phi, j}.$$This implies that $\phi\left(\mathbb{I} \right)[G_{\mathbb{Q}}]$-modules $\left.T_{\phi, j}\right/\varpi_{\phi}T_{\phi, j}$ are semi-simple for all $j=1, \cdots, l_{\phi}-1$. Thus $\left.T_{\phi}^{\mathrm{max}}\right/\varpi_{\phi}T_{\phi}^{\mathrm{max}}$ and $T_{\phi}^{\mathrm{min}}/\varpi_{\phi}T_{\phi}^{\mathrm{min}}$ must be not semi-simple $\phi\left(\mathbb{I} \right)[G_{\mathbb{Q}}]$-modules by Ribet's lemma (cf. \cite[Proposition 2.1]{Ri76}). We have the following exact sequence of $\phi\left(\mathbb{I} \right)[G_{\mathbb{Q}}]$-modules:
\begin{equation*}
0 \rightarrow \left.T_{\phi, 1}\right/\varpi_{\phi}T_{\phi}^{\mathrm{max}} \rightarrow \left.T_{\phi}^{\mathrm{max}}\right/\varpi_{\phi}T_{\phi}^{\mathrm{max}} \rightarrow \phi\left(\mathbb{I} \right)/\left(\varpi_{\phi} \right)\,\left(\mathbf{1} \right) \rightarrow 0.
\end{equation*}
Thus it must be non-split.
\end{proof}

Let $\mathbb{T}$ be a $G_{\mathbb{Q}}$-stable $\mathbb{I}$-free lattice of $\mathbb{V}_{\mathcal{F}}$. Let $\mathcal{T}^{\left(i \right)}=\mathbb{T}\hat{\otimes}_{\mathbb{Z}_p}\mathbb{Z}_p[[\Gamma]]\left(\tilde{\kappa}^{-1} \right)\otimes\omega^i$ and $\mathcal{A}=\mathcal{T}^{\left(i \right)}\otimes_{\mathcal{R}}\mathcal{R}^{\lor}$. Let $\mathcal{P}=\mathrm{Ker}\left(\phi\right)\mathcal{R}$ be a height-one prime ideal of $\mathcal{R}$ for $\phi \in \mathfrak{X}_{\mathrm{arith}}\left(\mathbb{I} \right)$. We define $\mathrm{Sel}_{\mathcal{A}[\mathcal{P}]}$ as follows:
\begin{equation}\label{1904091}
\mathrm{Sel}_{\mathcal{A}[\mathcal{P}]}=\mathrm{Ker}\left[H^{1}\left(\mathbb{Q}_{\Sigma}/\mathbb{Q}, \mathcal{A}[\mathcal{P}] \right) \rightarrow \displaystyle\prod_{l \in \Sigma\setminus\{\infty, p\}}H^{1}\left(I_l, \mathcal{A}[\mathcal{P}] \right)\times H^{1}\left(I_p, \left.\mathcal{A}[\mathcal{P}]\right/F^{+}\mathcal{A}[\mathcal{P}] \right) \right].
\end{equation}
Let $r_{s, \mathcal{A}}$ be the following map induced by $\mathcal{A}[\mathcal{P}] \hookrightarrow \mathcal{A}$:
\begin{equation}\label{res}
r_{s, \mathcal{A}} : \mathrm{Sel}_{\mathcal{A}[\mathcal{P}]} \rightarrow \mathrm{Sel}_{\mathcal{A}}[\mathcal{P}].
\end{equation}

Now we study the control theorem for Selmer groups in the sense of \cite[Proposition 5.1 and Proposition 5.2]{Ochiai06}. Note that \cite[Proposition 5.1 and Proposition 5.2]{Ochiai06} are proved under the assumption that the residue representation $\rho_{\mathcal{F}}\left(\mathfrak{m}_{\mathbb{I}} \right)$ is irreducible. We consider the residually reducible case. 

\begin{lem}\label{controlsel}
Let us keep the assumptions of Theorem \ref{yandongmain}. Then we have the following statements for a height-one prime ideal $\mathcal{P}$.
\begin{enumerate}
\renewcommand{\labelenumi}{(\arabic{enumi})}
\item Assume $i \not\equiv 0 \pmod{p-1}$ and $\omega^i \neq \chi^{-1}$. Then for any $G_{\mathbb{Q}}$-stable $\mathbb{I}$-free lattice $\mathbb{T}$, the map $r_{s, \mathcal{A}}$ is injective. 
\item Assume $i \equiv 0 \pmod{p-1}$ and that the ideal $\mathbb{J}$ is principal. Then when we choose $\mathbb{T}$ to be $\mathbb{T}^{\mathrm{max}}$ in \eqref{12011}, the map $r_{s, \mathcal{A}^{\mathrm{max}}}: \mathrm{Sel}_{\mathcal{A}^{\mathrm{max}}[\mathcal{P}]} \rightarrow \mathrm{Sel}_{\mathcal{A}^{\mathrm{max}}}[\mathcal{P}]$ is injective, where $\mathcal{A}^{\mathrm{max}}=\mathbb{T}^{\mathrm{max}}\hat{\otimes}_{\mathbb{Z}_p}\mathbb{Z}_p[[\Gamma]]\left(\tilde{\kappa}_{\mathrm{cyc}} \right)\otimes_{\mathcal{R}}\mathcal{R}^{\lor}$.
\item Assume further the condition (Tame) in Theorem \ref{910115}, then the map $r_{s, \mathcal{A}}$ is surjective for any $G_{\mathbb{Q}}$-stable $\mathbb{I}$-free lattice $\mathbb{T}$. 
\end{enumerate}
\end{lem}

\begin{proof}
Let us consider the following commutative diagram:
\small
\begin{equation}\label{190208}
\xymatrix{
    0 \ar[r] & \mathrm{Ker}\left(\iota^{\mathrm{Sel}}_{\mathcal{A}[\mathcal{P}], \mathcal{A}} \right) \ar[r] \ar[d] & \dfrac{H^{0}\left(\mathbb{Q}_{\Sigma}/\mathbb{Q}, \mathcal{A}\right)}{\mathcal{P} H^{0}\left(\mathbb{Q}_{\Sigma}/\mathbb{Q}, \mathcal{A} \right)} \ar@{^{(}-_>}[d] \\
    0 \ar[r]   & \mathrm{Sel}_{\mathcal{A}[\mathcal{P}]} \ar[r] \ar[d]^{r_{s, \mathcal{A}}} & H^1\left(\mathbb{Q}_{\Sigma}/\mathbb{Q}, \mathcal{A}[\mathcal{P}] \right) \ar[r] \ar[d]^{r_{h, \mathcal{A}}} & H^{1}\left(I_p, F^{-}\mathcal{A}[\mathcal{P}] \right)\times\displaystyle\prod_{l \in \Sigma\setminus\{p, \infty\}}H^{1}\left(I_l, \mathcal{A}[\mathcal{P}] \right) \ar[d] \\
    0 \ar[r] & \mathrm{Sel}_{\mathcal{A}}[\mathcal{P}] \ar[r] & H^1\left(\mathbb{Q}_{\Sigma}/\mathbb{Q}, \mathcal{A} \right)[\mathcal{P}] \ar[r] & H^{1}\left(I_p, F^{-}\mathcal{A} \right)[\mathcal{P}]\times\displaystyle\prod_{l \in \Sigma\setminus\{p, \infty \}}H^{1}\left(I_l, \mathcal{A} \right)[\mathcal{P}]. \\
   }
\end{equation}
\normalsize
We have the isomorphism $\dfrac{H^{0}\left(\mathbb{Q}_{\Sigma}/\mathbb{Q}, \mathcal{A}\right)}{\mathcal{P} H^{0}\left(\mathbb{Q}_{\Sigma}/\mathbb{Q}, \mathcal{A} \right)}\stackrel{\sim}{\rightarrow} \left(\left(\mathcal{T}^{\left(i \right), *} \right)_{G_{\mathbb{Q}}}[\mathcal{P}] \right)^{\lor}$. 

Since the first and the second assertions are proved in the same way, we prove the second assertion. By the above argument, it is sufficient to prove that the module $\left(\mathcal{T}^{\mathrm{max}, *} \right)_{G_{\mathbb{Q}}}$ is zero. By taking the base extension $\otimes_{\mathcal{R}}\left.\mathcal{R}\right/\mathfrak{M}_{\mathcal{R}}$, we have the surjection of $\mathcal{R}$-modules $$\left.\mathcal{T}^{\mathrm{max},*}\right/\mathfrak{M}_{\mathcal{R}}\mathcal{T}^{\mathrm{max},*} \twoheadrightarrow \left.\left(\mathcal{T}^{\mathrm{max},*} \right)_{G_{\mathbb{Q}}}\right/\mathfrak{M}_{\mathcal{R}}\left(\mathcal{T}^{\mathrm{max},*} \right)_{G_{\mathbb{Q}}}.$$ Since $\mathcal{T}^{\mathrm{max}}$ is free of rank two over $\mathcal{R}$, the $\left.\mathcal{R}\right/\mathfrak{M}_{\mathcal{R}}$-vector space $\left.\left(\mathcal{T}^{\mathrm{max},*} \right)_{G_{\mathbb{Q}}}\right/\mathfrak{M}_{\mathcal{R}}\left(\mathcal{T}^{\mathrm{max},*} \right)_{G_{\mathbb{Q}}}$ has dimension less than or equal to two. Assume that $\left.\left(\mathcal{T}^{\mathrm{max},*} \right)_{G_{\mathbb{Q}}}\right/\mathfrak{M}_{\mathcal{R}}\left(\mathcal{T}^{\mathrm{max},*} \right)_{G_{\mathbb{Q}}}$ has dimension two. Then $\left.\left(\mathcal{T}^{\mathrm{max},*} \right)_{G_{\mathbb{Q}}}\right/\mathfrak{M}_{\mathcal{R}}\left(\mathcal{T}^{\mathrm{max},*} \right)_{G_{\mathbb{Q}}}$ is isomorphic to $\left.\mathcal{T}^{\mathrm{max},*}\right/\mathfrak{M}_{\mathcal{R}}\mathcal{T}^{\mathrm{max},*}$. This contradicts to the fact that $\rho_{\mathcal{F}}\left(\mathfrak{m}_{\mathbb{I}} \right)$ is $G_{\mathbb{Q}}$-distinguished since the determinant $\mathrm{det}\,\rho_{\mathcal{F}}$ is an odd character. Assume that $\left.\left(\mathcal{T}^{\mathrm{max},*} \right)_{G_{\mathbb{Q}}}\right/\mathfrak{M}_{\mathcal{R}}\left(\mathcal{T}^{\mathrm{max},*} \right)_{G_{\mathbb{Q}}}$ has dimension one. Then $\left.\left(\mathcal{T}^{\mathrm{max},*} \right)_{G_{\mathbb{Q}}}\right/\mathfrak{M}_{\mathcal{R}}\left(\mathcal{T}^{\mathrm{max},*} \right)_{G_{\mathbb{Q}}}$ is a type $\mathbf{1}$ quotient of $\left.\mathcal{T}^{\mathrm{max},*}\right/\mathfrak{M}_{\mathcal{R}}\mathcal{T}^{\mathrm{max},*}$. By Lemma \ref{controllattice} we have the following non-spit exact sequence:
\begin{equation}\label{190329}
0 \rightarrow \left.\mathcal{R}\right/\mathfrak{M}_{\mathcal{R}}\ \left(\chi \right) \rightarrow \left.\mathcal{T}^{\mathrm{max}}\right/\mathfrak{M}_{\mathcal{R}}\mathcal{T}^{\mathrm{max}} \rightarrow \left.\mathcal{R}\right/\mathfrak{M}_{\mathcal{R}}\ \left(\mathbf{1} \right) \rightarrow 0.
\end{equation}
Since $\mathcal{T}^{\mathrm{max}}$ is a projective $\mathcal{R}$-module, we have the following isomorphism $$\mathrm{Hom}_{\left.\mathcal{R}\right/\mathfrak{M}_{\mathcal{R}}}\left(\left.\mathcal{T}^{\mathrm{max}}\right/\mathfrak{M}_{\mathcal{R}}\mathcal{T}^{\mathrm{max}}, \left.\mathcal{R}\right/\mathfrak{M}_{\mathcal{R}} \right) \stackrel{\sim}{\rightarrow} \left.\mathcal{T}^{\mathrm{max},*}\right/\mathfrak{M}_{\mathcal{R}}\mathcal{T}^{\mathrm{max},*}$$ by induction on the minimal number of generators of $\mathfrak{M}_{\mathcal{R}}$ (see \cite[Lemma A.12]{BP} for example). Then we have the following non-split exact sequence by \eqref{190329}:
\begin{equation*}
0 \rightarrow \left.\mathcal{R}\right/\mathfrak{M}_{\mathcal{R}}\ \left(\mathbf{1} \right) \rightarrow \left.\mathcal{T}^{\mathrm{max},*}\right/\mathfrak{M}_{\mathcal{R}}\mathcal{T}^{\mathrm{max},*} \rightarrow \left.\mathcal{R}\right/\mathfrak{M}_{\mathcal{R}}\ \left(\chi^{-1} \right) \rightarrow 0.
\end{equation*}
This implies that $\left.\mathcal{T}^{\mathrm{max},*}\right/\mathfrak{M}_{\mathcal{R}}\mathcal{T}^{\mathrm{max},*}$ has no type $\mathbf{1}$ quotient which contradicts to our assumption. Thus $\left.\left(\mathcal{T}^{\mathrm{max},*} \right)_{G_{\mathbb{Q}}}\right/\mathfrak{M}_{\mathcal{R}}\left(\mathcal{T}^{\mathrm{max},*} \right)_{G_{\mathbb{Q}}}$ is zero. Then $\left(\mathcal{T}^{\mathrm{max},*} \right)_{G_{\mathbb{Q}}}$ is zero by Nakayama's lemma. This completes the proof of the second assertion.

Now we prove the third assertion. We have that $\mathrm{Coker}\left(r_{s, \mathcal{A}} \right)$ is isomorphic to a sub-quotient $\left(U_{\phi} \right)^{\lor}$ by the same proof of \cite[Proposition 5.1 and Proposition 5.2]{Ochiai06}, where $$U_{\phi}=\displaystyle\bigoplus_{l\mid N}\left(\left(\mathbb{T}^{*} \right)_{I_l}[\mathcal{P}]\hat{\otimes}_{\mathbb{Z}_p}\mathbb{Z}_p[[\Gamma]]\left(\tilde{\kappa} \right)\otimes\omega^{-i} \right)^{D_p}.$$Under the assumption that $N$ is square free and $\chi$ is primitive, we have that the automorphic representation $\pi_{l}\left(f_{\phi} \right)$ associated to $f_{\phi}$ is principal series for every $l\mid N$ by \cite[Proposition 4.1.1]{Fuku}. Thus the image of $I_l$ under $G_{\mathbb{Q}} \rightarrow \mathrm{Aut}_{\mathbb{I}}\left(\mathbb{T} \right)$ is a finite group by the proof of \cite[Lemma 2.14]{FO12}. Then for every prime $l \mid N$, there exist a ring of integers $\mathcal{O}_l \subset \mathbb{I}$ of a finite extension of $\mathbb{Q}_p$ and integers $r_l, r^{\prime}_l \in \mathbb{Z}_{\geq 0}$ such that $\left(\mathbb{T}^{*}  \right)_{I_l}\cong \left.\mathbb{I}\right/\left(\varpi\right)^{r_l} \oplus \left.\mathbb{I}\right/\left(\varpi\right)^{r^{\prime}_l}$ by \cite[Theorem 3.3-(1)]{Ochiai06}, where $\varpi$ is a fixed uniformizer of $\mathcal{O}_l$. Thus we have that $\left(\mathbb{T}^{*}  \right)_{I_l}[\mathcal{P}]$ is trivial for every prime $l \mid N$. This completes the proof of Lemma \ref{controlsel}.

\end{proof}

Let $\mathbb{T}$ be a $G_{\mathbb{Q}}$-stable $\mathbb{I}$-free lattice and $\mathcal{A}=\mathcal{T}^{\left(i \right)}\otimes_{\mathcal{R}}\mathcal{R}^{\lor}$. Let $\mathcal{P}=\mathrm{Ker}\left(\phi\right)\mathcal{R}$ be a height-one prime ideal of $\mathcal{R}$ for $\phi \in \mathfrak{X}_{\mathrm{arith}}\left(\mathbb{I} \right)$. We define $H_{\mathrm{ur}}^1\left(D_l, \mathcal{A} \right)$ and $H_{\mathrm{ur}}^1\left(D_l, \mathcal{A}[\mathcal{P}] \right)$ for every prime $l\mid N$ as follows:
\begin{align*}
&H_{\mathrm{ur}}^1\left(D_l, \mathcal{A} \right):=\mathrm{Ker}\left[H^1\left(D_l, \mathcal{A} \right) \rightarrow H^1\left(I_l, \mathcal{A} \right) \right],\\
&H_{\mathrm{ur}}^1\left(D_l, \mathcal{A}[\mathcal{P}] \right):=\mathrm{Ker}\left[H^1\left(D_l, \mathcal{A}[\mathcal{P}] \right) \rightarrow H^1\left(I_l, \mathcal{A}[\mathcal{P}] \right) \right].
\end{align*}
We define $H_{\mathrm{Gr}}^1\left(D_p, \mathcal{A} \right)$ and $H_{\mathrm{Gr}}^1\left(D_p, \mathcal{A}[\mathcal{P}] \right)$ as follows:
\begin{align*}
&H_{\mathrm{Gr}}^1\left(D_p, \mathcal{A} \right):=\mathrm{Ker}\left[H^1\left(D_p, \mathcal{A} \right) \rightarrow H^{1}\left(I_p, \left.\mathcal{A}\right/F^{+}\mathcal{A} \right) \right],\\
&H_{\mathrm{Gr}}^1\left(D_p, \mathcal{A}[\mathcal{P}] \right):=\mathrm{Ker}\left[H^1\left(D_p, \mathcal{A}[\mathcal{P}] \right) \rightarrow H^{1}\left(I_p, \left.\mathcal{A}[\mathcal{P}]\right/F^{+}\mathcal{A}[\mathcal{P}] \right) \right].
\end{align*}
We introduce the following lemma:
\begin{lem}[Ochiai]\label{190207}
Let us keep the assumptions of Theorem \ref{yandongmain}. Then under the assumption that $\mathcal{R}$ is Gorenstein, the following localization maps 
\begin{equation*}
H^{1}\left(\mathbb{Q}_{\Sigma}/\mathbb{Q}, \mathcal{A} \right) \rightarrow \dfrac{H^1\left(D_p, \mathcal{A} \right)}{H_{\mathrm{Gr}}^1\left(D_p, \mathcal{A} \right)} \oplus \displaystyle\bigoplus_{l\mid N}\dfrac{H^1\left(D_l, \mathcal{A} \right)}{H_{\mathrm{ur}}^1\left(D_l, \mathcal{A} \right)}
\end{equation*}
and
\begin{equation*}
H^{1}\left(\mathbb{Q}_{\Sigma}/\mathbb{Q}, \mathcal{A}[\mathcal{P}] \right) \rightarrow \dfrac{H^1\left(D_p, \mathcal{A}[\mathcal{P}] \right)}{H_{\mathrm{Gr}}^1\left(D_p, \mathcal{A}[\mathcal{P}] \right)} \oplus \displaystyle\bigoplus_{l\mid N}\dfrac{H^1\left(D_l, \mathcal{A}[\mathcal{P}] \right)}{H_{\mathrm{ur}}^1\left(D_l, \mathcal{A}[\mathcal{P}] \right)}
\end{equation*}
are surjective.
\end{lem}
Note that as is mentioned in \cite[Lemma 3.3]{Ochiai08}, the above lemma is proved in the same way as \cite[Corollary 4.12]{Ochiai06}. Although \cite[Corollary 4.12]{Ochiai06} is proved under the assumption that $\rho_{\mathcal{F}}\left(\mathfrak{m}_{\mathbb{I}} \right)$ is irreducible, the control theorem of Bloch-Kato's Selmer groups (cf. \cite[Theorem 2.4]{Ochiai00}) makes the proof feasible in the residually reducible case.

By Lemma \ref{190207}, we deduce the following lemma which is proved in the same way as \cite[Lemma 7.2]{Ochiai06}.

\begin{lem}[Ochiai {\cite[Lemma 7.2]{Ochiai06}}]\label{190110}
Let us keep the assumptions of Theorem \ref{yandongmain}. Let $\left(\mathrm{Sel}_{\mathcal{A}} \right)_{\mathrm{null}}^{\lor}$ the maximal pseudo-null $\mathcal{R}$-submodule of $\left(\mathrm{Sel}_{\mathcal{A}} \right)^{\lor}$. Let $\mathcal{P}=\mathrm{Ker}\left(\phi\right)\mathcal{R}$ be a height-one prime ideal for some $\phi \in \mathfrak{X}_{\mathrm{arith}}\left(\mathbb{I} \right)$. Assume that $N$ is square-free, then $\left(\mathrm{Sel}_{\mathcal{A}} \right)_{\mathrm{null}}^{\lor}/\mathcal{P}\left(\mathrm{Sel}_{\mathcal{A}} \right)_{\mathrm{null}}^{\lor}$ is a pseudo-null $\left.\mathcal{R}\right/\mathcal{P}$-module. 
\end{lem}
In \cite{Ochiai06}, we deduce \cite[Lemma 7.2]{Ochiai06} from \cite[Corollary 4.12]{Ochiai06} in the situation when the residual representation is irreducible. The exactly same argument works even if we remove the assumption of the residually irreducibility. Note that the assumption that $N$ is square-free implies that the module $U_{\phi}$ in the proof of Proposition \ref{controlsel} is trivial.

When $N=1$, we have the following result by Ochiai:
\begin{prp}[Ochiai {\cite[Proposition 8.1]{Ochiai06}}]\label{Ochiainonps}
Suppose that $\mathbb{I}$ is a regular local ring. Let $\mathbb{T}$ be a $G_{\mathbb{Q}}$-stable $\mathbb{I}$-free lattice and $\mathcal{A}=\mathcal{T}^{\left(i \right)}\otimes_{\mathcal{R}}\mathcal{R}^{\lor}$. Assume $N=1$, then $\left(\mathrm{Sel}_{\mathcal{A}} \right)^{\lor}$ has no non-trivial pseudo-null $\mathcal{R}$-submodule. 

\end{prp}

Under the above preparation, we return to the proof of Theorem \ref{910115}.
\begin{proof}[Proof of Theorem \ref{910115}]
Let $\mathbb{T}^{\mathrm{max}}$ be the lattice in \eqref{12011}. By the proof of Theorem \ref{yandongmain}-(3), we have 
\begin{equation}\label{10284}
\left(L_p^{\mathrm{alg}}\left(\mathcal{T}^{\mathrm{max}, \left(0 \right)} \right) \right)=\left(L_p^{\mathrm{alg}}\left(\mathcal{T}^{\mathrm{min}, \left(0 \right)} \right) \right)\mathcal{J}^{**}.
\end{equation}
Then under the assumption ($p$Four), we have 
\begin{equation}\label{1903291}
\left(L_p^{\mathrm{alg}}\left(\mathcal{T}^{\mathrm{max}, \left(0 \right)} \right) \right)\subset\left(a\left(p, \mathcal{F} \right)-1 \right).
\end{equation}
Let $\mathcal{P}=\mathrm{Ker}\left(\phi\right)\mathcal{R}$ for some $\phi \in \mathfrak{X}_{\mathrm{arith}}\left(\mathbb{I} \right)$. By Lemma \ref{controlsel}, the following map 
\begin{equation*}
r_{s, \mathcal{A}^{\mathrm{max}}}: \mathrm{Sel}_{\mathcal{A}^{\mathrm{max}}[\mathcal{P}]} \rightarrow \mathrm{Sel}_{\mathcal{A}^{\mathrm{max}}}[\mathcal{P}]
\end{equation*}
is an isomorphism. Then under the assumption (Tame), the image of $\mathrm{char}_{\mathcal{R}}\left(\mathrm{Sel}_{\mathcal{A}^{\mathrm{max}}}\right)^{\lor}$ under $\mathcal{R} \twoheadrightarrow \left.\mathcal{R}\right/\mathcal{P}$ is equal to $\mathrm{char}_{\left.\mathcal{R}\right/\mathcal{P}}\left(\mathrm{Sel}_{\mathcal{A}^{\mathrm{max}}[\mathcal{P}]}  \right)^{\lor}$ by Lemma \ref{190110} and Proposition \ref{Ochiainonps}. 

Now we study $\mathrm{char}_{\left.\mathcal{R}\right/\mathcal{P}}\left(\mathrm{Sel}_{\mathcal{A}^{\mathrm{max}}[\mathcal{P}]}  \right)^{\lor}$. Let $T_{\phi}^{\mathrm{max}}=\mathbb{T}\otimes_{\mathbb{I}}\phi\left(\mathbb{I}\right)$ and $\widetilde{T_{\phi}}^{\mathrm{max}, \left(0\right)}=T_{\phi}^{\mathrm{max}}\otimes_{\mathbb{Z}_p}\mathbb{Z}_p[[\Gamma]]\left(\tilde{\kappa}^{-1} \right)$. By Lemma \ref{controllattice} we have that $\left.\mathcal{T}^{\mathrm{max}, \left(0 \right)}\right/\mathcal{P}\mathcal{T}^{\mathrm{max}, \left(0 \right)}$ is isomorphic to $\widetilde{T_{\phi}}^{\mathrm{max}, \left(0\right)}$. Under the assumption ($D_p$-dist), we have that $F^{+}\mathbb{T}^{\mathrm{max}}$ is a direct summand of $\mathbb{T}^{\mathrm{max}}$ by Lemma \ref{FO}. Thus $\mathrm{Sel}_{\mathcal{A}^{\mathrm{max}}[\mathcal{P}]}$ is isomorphic to $\mathrm{Sel}_{\widetilde{A}^{\mathrm{max}}}$, where $\widetilde{A}^{\mathrm{max}}=\widetilde{T_{\phi}}^{\mathrm{max}, \left(0\right)}\otimes_{\Lambda_{\phi\left(\mathbb{I}\right)}^{\mathrm{cyc}}}{\Lambda_{\phi\left(\mathbb{I}\right)}^{\mathrm{cyc}}}^{\lor}$. Then by Proposition \ref{RGde}, Lemma \ref{controllattice} and a calculation of Bella\"iche-Pollack \cite[Theorem 5.12]{BP}, we have the following equality 
\begin{equation}\label{10281}
\mathrm{char}_{\left.\mathcal{R}\right/\mathcal{P}}\left(\mathrm{Sel}_{\mathcal{A}^{\mathrm{max}}[\mathcal{P}]}  \right)^{\lor}=\left(a\left(p, f_{\phi} \right)-1\right).
\end{equation}
Thus by \eqref{1903291}, we have 
\begin{equation*}
\left(L_p^{\mathrm{alg}}\left(\mathcal{T}^{\mathrm{max}, \left(0 \right)} \right) \right)=\left(a\left(p, \mathcal{F} \right)-1 \right)=\mathcal{J}.
\end{equation*}
Thus $L_p^{\mathrm{alg}}\left(\mathcal{T}^{\mathrm{min}, \left(0 \right)}\right)$ is a unit of $\mathcal{R}$ follows by \eqref{10284}. This completes the proof of Theorem \ref{910115}. 
\end{proof}

We give an example at the end of this paper. Let $\Delta \in S_{12}\left(\mathrm{SL}_2\left(\mathbb{Z} \right) \right)$ be the Ramanujan's cusp form whose $q$-expansion is equal to $q\displaystyle\prod_{n=1}\left(1-q^n \right)^{24}$. Assume that $\Delta$ is $p$-ordinary (the only known primes at which $\Delta$ is non-ordinary are $p=2, 3, 5, 7, 2411$ and $7758337633$). We have that $S_{12}\left(\mathrm{SL}_2\left(\mathbb{Z} \right) \right)$ is a rank one $\mathbb{Z}_p$-module. Then by Hida's control theorem of Hecke algebra, the condition ($\Lambda$) holds for $\mathcal{F}_{\Delta}$ and there exists a unique $\Lambda$-adic normalized eigen cusp form $\mathcal{F}_{\Delta}=\mathcal{F}_{\Delta}\left(X, q \right) \in \Lambda[[q]]$ such that $\mathcal{F}_{\Delta}\left(u^{10}, q \right)$ is the $p$-stabilization of $\Delta$ (cf. \cite[\S 7.6]{H3}). By abuse of notation, we denote by the same symbol $\Delta$ for its $p$-stabilization.

The only one prime $p$ such that the residue representation of $\rho_{\mathcal{F}_{\Delta}}\left(\mathfrak{m}_{\Lambda}\right)$ is reducible is equal to $691$, which is fixed from now on to the end of this section. It is well-known that the Kubota-Leopoldt $p$-adic $L$-function $\mathcal{L}_p\left(\omega^{11}; \gamma^{\prime} \right)$ is a prime ideal of $\Lambda$. In our previous paper \cite[\S 4]{Y}, we proved $\sharp\mathscr{L}\left(  \rho_{\mathcal{F}_{\Delta}} \right)=\infty$. However by our main theorem, we have $\sharp\mathscr{L}^{\mathrm{fr}}\left(  \rho_{\mathcal{F}_{\Delta}} \right)=2$, i.e. there are exactly two $G_{\mathbb{Q}}$-stable $\Lambda$-free lattices $\mathbb{T}^{\mathrm{min}}$ and $\mathbb{T}^{\mathrm{max}}$ up to $G_{\mathbb{Q}}$-isomorphism. Moreover we have 
\begin{equation*}
\mathscr{L}_p^{\mathrm{alg}}\left(\rho_{\mathcal{F}_{\Delta}}^{\mathrm{n.ord}, \left(i \right)} \right)=
\begin{cases}
\left\{ L_p^{\mathrm{alg}}\left(\mathcal{T}^{\mathrm{min}, \left(i \right)} \right) \right\}& \left(i: \text{odd} \right) \\
\left\{ L_p^{\mathrm{alg}}\left(\mathcal{T}^{\mathrm{min}, \left(i \right)} \right), L_p^{\mathrm{alg}}\left(\mathcal{T}^{\mathrm{max}, \left(i \right)} \right) \right\}& \left(i: \text{even} \right).
\end{cases}
\end{equation*}
Now let $i \equiv 0 \pmod{p-1}$. In \cite[Appendix I]{M11}, Mazur computed the following equality of ideals $$\left(\mathcal{L}_p\left(\omega^{11}; \gamma^{\prime} \right)\right)=\left(a\left(p, \mathcal{F} \right)-1 \right).$$ Thus by Theorem \ref{910115}, we have that $L_p^{\mathrm{alg}}\left(\mathcal{T}^{\mathrm{min}, \left(0 \right)} \right)$ is a unit of $\mathcal{R}$ and $L_p^{\mathrm{alg}}\left(\mathcal{T}^{\mathrm{max}, \left(0 \right)} \right)=\mathcal{L}_p\left(\omega^{11}; \gamma^{\prime} \right)$ up to multiplying by elements of $\mathcal{R}^{\times}$.

Thus we calculated the set $\mathscr{L}_p^{\mathrm{alg}}\left(\rho_{\mathcal{F}_{\Delta}}^{\mathrm{n.ord}, \left(i \right)} \right)$ of two-variable algebraic $p$-adic $L$-functions when $i \equiv 0 \pmod{p-1}$. In a forthcoming paper, we will calculate $\mathscr{L}_p^{\mathrm{alg}}\left(\rho_{\mathcal{F}_{\Delta}}^{\mathrm{n.ord}, \left(i \right)} \right)$ for more general $i$.

\appendix
\section{The graph structure on $\mathscr{C}^{\mathrm{fr}}\left(\rho_{\mathcal{F}} \right)$}
Let $A$ be a complete discrete valuation ring and $\rho: G \rightarrow \mathrm{GL}_2\left(A \right)$ a continuous representation of a compact group $G$. Then one can define a graph structure on $\mathscr{C}\left(\rho \right)$ in the sense of Serre \cite[Chap. II, \S 1]{SeTe} in which we have that the graph $\mathscr{C}\left(\rho \right)$ is a tree. Then the tree $\mathscr{C}\left(\rho \right)$ was studied more deeply by Bella\"iche-Chenevier \cite{BC14}. 

In this appendix, we study the case when $\rho$ comes from a Hida deformation i.e. $A$ becomes a domain of Krull dimension two, which is not a discrete valuation ring. We give a graph structure on $\mathscr{C}^{\mathrm{fr}}\left(\rho_{\mathcal{F}} \right)$ by using the arguments in \S 4.3. Recall that we have $\sharp\mathscr{L}^{\mathrm{fr}}\left(\rho_{\mathcal{F}} \right)=\sharp\mathscr{C}^{\mathrm{fr}}\left(\rho_{\mathcal{F}} \right)$ by Theorem \ref{yandongmain}-(1). Thus the graph structure enables us to understand the set of isomorphic classes of $G_{\mathbb{Q}}$-stable free lattices more concretely. We keep the assumptions and the notation of Theorem \ref{yandongmain} throughout this section. First we give some lemmas as preparation.
\begin{lem}\label{neighbor}
Let us keep the assumptions and the notation of Theorem \ref{yandongmain}. Let $\mathbb{T}^{\prime} \subset \mathbb{T}$ be $G_{\mathbb{Q}}$-stable $\mathbb{I}$-free lattices of $\mathbb{V}_{\mathcal{F}}$ such that $\left.\mathbb{T}\right/{\mathbb{T}^{\prime}} \cong \left.\mathbb{I}\right/{\mathfrak{p}}$ for some $\mathfrak{p} \in P^1\left(\mathbb{I}\right)$. Then $\left.\mathbb{T}^{\prime}\right/{\mathfrak{p}\mathbb{T}} \cong \left.\mathbb{I}\right/{\mathfrak{p}}$. 
\end{lem}

\begin{proof}
Let us consider the following commutative diagram
\begin{equation}\label{Keycom1220}
 \xymatrix{
   0 \ar[r]   & F^{+}\mathbb{T}^{\prime} \ar[r] \ar@{^{(}-_>}[d] & \mathbb{T}^{\prime} \ar[r] \ar@{^{(}-_>}[d] & F^{-}\mathbb{T}^{\prime} \ar[r] \ar[d] & 0 \\
    0 \ar[r] & F^{+}\mathbb{T} \ar[r] & \mathbb{T} \ar[r] & F^{-}\mathbb{T} \ar[r] & 0.
  }
\end{equation}
By combining the assumption ($D_p$-dist) and the snake lemma, we have 
\begin{equation*}
\mathrm{Ker}\left(F^{-}\mathbb{T}^{\prime} \rightarrow F^{-}\mathbb{T}   \right)=0
\end{equation*}
and the following exact sequence
\begin{equation}\label{12201}
0 \rightarrow \left.F^{+}\mathbb{T}\right/{F^{+}\mathbb{T}^{\prime}} \rightarrow \left.\mathbb{T}\right/{\mathbb{T}^{\prime}} \rightarrow \left.F^{-}\mathbb{T}\right/{F^{-}\mathbb{T}^{\prime}} \rightarrow 0.
\end{equation}
Since $\left.\mathbb{T}\right/{\mathbb{T}^{\prime}}$ is a cyclic $\mathbb{I}$-module, by the assumption ($D_p$-dist), we must have $F^{+}\mathbb{T}=F^{+}\mathbb{T}^{\prime}$ or $F^{-}\mathbb{T}=F^{-}\mathbb{T}^{\prime}$. 

Let us consider the condition when $F^{+}\mathbb{T}=F^{+}\mathbb{T}^{\prime}$. Then $\left.F^{-}\mathbb{T}\right/{F^{-}\mathbb{T}^{\prime}}$ is isomorphic to $\left.\mathbb{I}\right/{\mathfrak{p}}$ by \eqref{12201}. Recall that the $\mathbb{I}$-modules $F^{+}\mathbb{T}, F^{-}\mathbb{T}, F^{+}\mathbb{T}^{\prime}$ and $F^{-}\mathbb{T}^{\prime}$ are free of rank one by Lemma \ref{FO}. Then $\mathfrak{p}F^{-}\mathbb{T}=F^{-}\mathbb{T}^{\prime}$. Thus by the following commutative diagram
\begin{equation*}
 \xymatrix{
    0 \ar[r] & \mathfrak{p}F^{+}\mathbb{T} \ar[r] \ar@{^{(}-_>}[d] & \mathfrak{p}\mathbb{T} \ar[r] \ar@{^{(}-_>}[d] & \mathfrak{p}F^{-}\mathbb{T} \ar[r] \ar@{=}[d] & 0 \\
    0 \ar[r] & F^{+}\mathbb{T}^{\prime} \ar[r] & \mathbb{T}^{\prime} \ar[r] & F^{-}\mathbb{T}^{\prime} \ar[r] & 0,
  }
\end{equation*}
we have 
\begin{equation*}
\left.\mathbb{T}^{\prime}\right/{\mathfrak{p}\mathbb{T}}\stackrel{\sim}{\rightarrow} \left.F^{+}\mathbb{T}^{\prime}\right/{\mathfrak{p}F^{+}\mathbb{T}} \stackrel{\sim}{\rightarrow} \left.\mathbb{I}\right/{\mathfrak{p}}.
\end{equation*}
When $F^{-}\mathbb{T}=F^{-}\mathbb{T}^{\prime}$, by the same argument, we have 
\begin{equation*}
\left.\mathbb{T}^{\prime}\right/{\mathfrak{p}\mathbb{T}}\stackrel{\sim}{\rightarrow} \left.F^{-}\mathbb{T}^{\prime}\right/{\mathfrak{p}F^{-}\mathbb{T}}\stackrel{\sim}{\rightarrow}\left.\mathbb{I}\right/{\mathfrak{p}}.
\end{equation*}
This completes the proof.
\end{proof}
\begin{lem}\label{neighbor well def}
Let us keep the assumptions and the notation of Theorem \ref{yandongmain}. Let $\mathbb{T}^{\prime} \subset \mathbb{T}$ be $G_{\mathbb{Q}}$-stable $\mathbb{I}$-free lattices such that $\left.\mathbb{T}\right/{\mathbb{T}^{\prime}} \stackrel{\sim}{\rightarrow} \left.\mathbb{I}\right/{\mathfrak{p}}$ for some $\mathfrak{p} \in P^1\left(\mathbb{I} \right)$. Let $x, y$ be elements of $\mathbb{K}^{\times}$ and $\mathfrak{a}$ a principal integral ideal of $\mathbb{I}$. We have the following statements.
\begin{enumerate}
\item[(1)]Suppose $y\mathbb{T}^{\prime} \subset x\mathbb{T}$ and $\left.x\mathbb{T}\right/{y\mathbb{T}^{\prime}} \stackrel{\sim}{\rightarrow} \left.\mathbb{I}\right/{\mathfrak{a}}$, then $\mathfrak{a}=\mathfrak{p}$ and $\left(x \right)=\left(y \right)$ as fractional ideals.
\item[(2)]Suppose $x\mathbb{T} \subset y\mathbb{T}^{\prime}$ and $\left.y\mathbb{T}^{\prime}\right/{x\mathbb{T}} \stackrel{\sim}{\rightarrow} \left.\mathbb{I}\right/{\mathfrak{a}}$, then $\mathfrak{a}=\mathfrak{p}$ and $\left(x \right)=\left(y \right)\mathfrak{p}$ as fractional ideals.
\end{enumerate}

\end{lem}
\begin{proof}
We prove the first assertion. By the proof of Lemma \ref{neighbor}, we have that $\mathbb{T}$ and $\mathbb{T}^{\prime}$ satisfy one of the following equalities.
\begin{enumerate}
\item[(i)]$F^{+}\mathbb{T}^{\prime}=F^{+}\mathbb{T}, F^{-}\mathbb{T}^{\prime}=\mathfrak{p}F^{-}\mathbb{T}$
\item[(ii)]$F^{+}\mathbb{T}^{\prime}=\mathfrak{p}F^{+}\mathbb{T}, F^{-}\mathbb{T}^{\prime}=F^{-}\mathbb{T}$
\end{enumerate}
Since the proof under the conditions (i) and (ii) are done in the same way. We only prove under (i). Under the assumption, $x\mathbb{T}$ and $y\mathbb{T}^{\prime}$ satisfy one of the following equalities.
\begin{enumerate}
\item[$(\mathrm{i})^{\prime}$]$F^{+}\left(y\mathbb{T}^{\prime}\right)=F^{+}\left(x\mathbb{T} \right), F^{-}\left(y\mathbb{T}^{\prime}\right)=\mathfrak{a}F^{-}\left(x\mathbb{T} \right)$
\item[$(\mathrm{ii})^{\prime}$]$F^{+}\left(y\mathbb{T}^{\prime}\right)=\mathfrak{a}F^{+}\left(x\mathbb{T} \right), F^{-}\left(y\mathbb{T}^{\prime}\right)=F^{-}\left(x\mathbb{T} \right)$
\end{enumerate}
Assume the condition $(\mathrm{ii})^{\prime}$. Since $F^{+}\mathbb{T}, F^{-}\mathbb{T}, F^{+}\mathbb{T}^{\prime}$ and $F^{-}\mathbb{T}^{\prime}$ are free $\mathbb{I}$-modules of rank one, we have $\mathbb{I}=\mathfrak{ap}$ by (i). This contradicts to the assumption that $\mathfrak{a}$ is an integral ideal. Thus $x\mathbb{T}$ and $y\mathbb{T}^{\prime}$ must satisfy the condition $(\mathrm{i})^{\prime}$. Then we have $\left(x \right)=\left(y \right)$ and $\mathfrak{a}=\mathfrak{p}$ by (i). 

We prove the second assertion. We have $\left.\mathbb{T}^{\prime}\right/{\mathfrak{p}\mathbb{T}} \stackrel{\sim}{\rightarrow} \left.\mathbb{I}\right/{\mathfrak{p}}$ by Lemma \ref{neighbor}. Since $\mathfrak{p}$ is principal, we may replace $\mathbb{T}$ to $\mathfrak{p}\mathbb{T}$. Then the second assertion follows by the first assertion.

\end{proof}

For a $G_{\mathbb{Q}}$-stable $\mathbb{I}$-free lattice $\mathbb{T}$, we denote by $\left[\mathbb{T} \right]:=\set{\alpha\mathbb{T} | \alpha \in \mathbb{K}^{\times} }$ its homothetic class. For two points $x, x^{\prime} \in \mathscr{C}^{\mathrm{fr}}\left(\rho_{\mathcal{F}} \right)$, we say that $x$ is a \textit{neighbor} of $x^{\prime}$ if there exist representatives $\mathbb{T}$ and $\mathbb{T}^{\prime}$ of $x$ and $x^{\prime}$ respectively such that $\left.\mathbb{T}\right/{\mathbb{T}^{\prime}} \stackrel{\sim}{\rightarrow} \left.\mathbb{I}\right/{\mathfrak{p}}$ for some $\mathfrak{p} \in P^{1}\left(\mathbb{I} \right)$. The definition of neighbor is well-defined by Lemma \ref{neighbor well def} and is symmetric by Lemma \ref{neighbor}.

We define a graph structure on $\mathscr{C}^{\mathrm{fr}}\left(\rho_{\mathcal{F}} \right)$ as follows: 
\begin{enumerate}
\item[(vert)]The vertex set is $\mathscr{C}^{\mathrm{fr}}\left(\rho_{\mathcal{F}} \right)$.
\item[(edge)]For two points $x, x^{\prime} \in \mathscr{C}^{\mathrm{fr}}\left(\rho_{\mathcal{F}} \right)$, we draw an edge from $x$ to $x^{\prime}$ if $x$ is a neighbor of $x^{\prime}$.
\end{enumerate}

Now we study the graph $\mathscr{C}^{\mathrm{fr}}\left(\rho_{\mathcal{F}} \right)$. Recall that $\mathbb{J}$ is the ideal of $\mathbb{I}$ which is generated by $a\left(l, \mathcal{F} \right)-1-\chi\left(l \right)\langle l \rangle{\kappa^{\prime}}^{-1}\left(\langle l \rangle \right)$ for all primes $l \nmid Np$. When $\mathbb{J}^{**}=\mathbb{I}$, the graph $\mathscr{C}^{\mathrm{fr}}\left(\rho_{\mathcal{F}} \right)$ reduces to one point by Lemma \ref{1108}. Now we consider the case when $\mathbb{J}^{**} \subsetneq \mathbb{I}$. We recall the notation of \S 4. Let 
\begin{align*}
\mathcal{N}:&=\set{\mathfrak{p} \in P^{1}\left(\mathbb{I} \right) | \mathrm{ord}_{\mathfrak{p}}\mathbb{J}^{**} \neq 0} \\
&=\set{\mathfrak{p}_1, \cdots, \mathfrak{p}_r}
\end{align*} 
and let $n_i=\mathrm{ord}_{\mathfrak{p}_i}\mathbb{J}^{**}$ $\left(1 \leq i \leq r\right)$. We define the graph of $r$-dimensional rectangle $\mathrm{Rect}_{\left(0, \cdots, 0\right)}^{\left(n_1, \cdots, n_r \right)}$ as follows.
\begin{enumerate}
\item[(vert)]The vertex set is $\left\{\left(j_1, \cdots, j_r \right) \in \mathbb{Z}^{r} |\ 0 \leq j_1 \leq n_1, \cdots, 0 \leq j_r \leq n_r \right\}.$
\item[(edge)]We draw an edge from $\left(j_1, \cdots, j_r \right)$ to $\left(j_1^{\prime}, \cdots, j_r^{\prime} \right)$ if there exists an integer $1 \leq s \leq r$ such that 
\begin{equation}\label{12231}
|j_i^{\prime}-j_i|=\begin{cases}
0 & \left(i \neq s \right) \\
1 & \left(i=s \right). \\
\end{cases}
\end{equation}
\end{enumerate}

\begin{prp}
Let us keep the assumptions of Theorem \ref{yandongmain}. Then the graph $\mathscr{C}^{\mathrm{fr}}\left(\rho_{\mathcal{F}} \right)$ is isomorphic to $\mathrm{Rect}_{\left(0, \cdots, 0\right)}^{\left(n_1, \cdots, n_r \right)}$.
\end{prp}
\begin{proof}
For each $\mathfrak{p}_i \in \mathcal{N}$ and $0 \leq j_i \leq n_i$, let $\mathbb{T}\left(j_1, \cdots, j_r \right)$ be the $G_{\mathbb{Q}}$-stable $\mathbb{I}$-free lattice which is defined in \S 4.3. Then we have 
\begin{equation*}
\mathscr{C}^{\mathrm{fr}}\left(\rho_{\mathcal{F}} \right)=\left\{\left[\mathbb{T}\left(j_1, \cdots, j_r \right)\right] |\ 0 \leq j_1 \leq n_1, \cdots, 0 \leq j_r \leq n_r \right\}
\end{equation*}
by Lemma \ref{1108}.

Let us take elements $\left(j_1, \cdots, j_r \right), \left(j_1^{\prime}, \cdots, j_r^{\prime} \right) \in \mathrm{vert}\,\mathrm{Rect}_{\left(0, \cdots, 0\right)}^{\left(n_1, \cdots, n_r \right)}$ which satisfy \eqref{12231}. We may assume $j_s^{\prime}-j_s=1$. Then we have $$\mathbb{T}\left(0, \cdots, 0\right) \subset \mathbb{T}\left(j_1, \cdots, j_r \right) \subset \mathbb{T}\left(j_1^{\prime}, \cdots, j_r^{\prime} \right).$$Hence we have
\begin{equation*}
\left.\mathbb{T}\left(j_1^{\prime}, \cdots, j_r^{\prime} \right)\right/{\mathbb{T}\left(j_1, \cdots, j_r \right)} \stackrel{\sim}{\rightarrow} \left.\mathbb{I}\right/{\mathfrak{p}_s}
\end{equation*}
by Lemma \ref{Fil2}. This implies that $\left[\mathbb{T}\left(j_1^{\prime}, \cdots, j_r^{\prime} \right)   \right]$ is a neighbor of $\left[\mathbb{T}\left(j_1, \cdots, j_r \right)    \right]$ and we complete the proof. 
\end{proof}

For example, when $\mathbb{J}^{**}=\mathfrak{p}^e$, the graph $\mathscr{C}^{\mathrm{fr}}\left(\rho_{\mathcal{F}} \right)$ is a segment. When $\mathbb{J}^{**}=\mathfrak{pqr}$ with $\mathfrak{p}\neq\mathfrak{q}\neq\mathfrak{r} \in P^1\left(\mathbb{I} \right)$, the graph $\mathscr{C}^{\mathrm{fr}}\left(\rho_{\mathcal{F}} \right)$ becomes a cube.


\begin{thebibliography}{9}




\bibitem{Bell2}
J. Bella\"iche, G. Chenevier,
\textit{Lissit\'e de la courbe de Hecke de GL2 aux points Eisenstein critiques.},
J. Inst. Math. Jussieu 5 (2006), no. 2, 333-349.


\bibitem{BC14}
J. Bella\"iche, G. Chenevier,
\textit{Sous-groupes de $GL_2$ et arbres.},
J. Algebra 410 (2014), 501-525.







\bibitem{BP}
J. Bella\"iche, R. Pollack,
\textit{Congruences with Eisenstein series and $\mu$-invariants.},
Compositio Math. 155 (2019) 863-901.

\bibitem{Bour}
N. Bourbaki,
\textit{Commutative algebra. Chapters 1-7.},
Translated from the French. Reprint of the 1989 English translation. Elements of Mathematics (Berlin). Springer-Verlag, Berlin, 1998.



\bibitem{FO12}
O. Fouquet, T. Ochiai,
\textit{Control theorems for Selmer groups of nearly ordinary deformations.},
J. Reine Angew. Math. 666 (2012), 163-187.

\bibitem{Fuku}
K. Fukunaga,
\textit{Triple product $p$-adic $L$-function attached to $p$-adic families of modular forms.},
arXiv:1909. 03165.

\bibitem{RG89}
R. Greenberg,
\textit{Iwasawa theory for $p$-adic representations.},
Algebraic number theory, 97-137, Adv. Stud. Pure Math., 17, Academic Press, Boston, MA, 1989.



\bibitem{RG}
R. Greenberg,
\textit{Iwasawa theory and $p$-adic deformations of motives.},
Motives (Seattle, WA, 1991), 193-223, Proc. Sympos. Pure Math., 55, Part 2, Amer. Math. Soc., Providence, RI, 1994.









\bibitem{H1}
H. Hida
\textit{Iwasawa modules attached to congruences of cusp forms.},
Ann. Sci. \'Ecole Norm. Sup. (4) 19 (1986), no. 2, 231-273.



\bibitem{H2}
H. Hida,
\textit{Galois representations into $\mathrm{GL}_2(\mathbb{Z}_p[[X]])$ attached to ordinary cusp forms.},
Invent. Math. 85 (1986), no. 3, 545-613. 

\bibitem{H3}
H. Hida,
\textit{Elementary theory of $L$-functions and Eisenstein series.},
London Mathematical Society Student Texts, 26. Cambridge University Press, Cambridge, 1993.

\bibitem{Iw68}
K. Iwasawa,
\textit{On $p$-adic $L$-functions.},
Ann. of Math. (2) 89 1969 198-205. 






\bibitem{M11}
B. Mazur,
\textit{How can we construct abelian Galois extensions of basic number fields ?}.,
Bull. Amer. Math. Soc. (N.S.) 48 (2011), no. 2, 155-209.

\bibitem{MW84}
B. Mazur, A. Wiles,
\textit{Class fields of abelian extensions of $\mathbb{Q}$.},
Invent. Math. 76 (1984), no. 2, 179-330.



\bibitem{MW2}
B. Mazur, A. Wiles,
\textit{On $p$-adic analytic families of Galois representations.},
Compositio Math. 59 (1986), no. 2, 231-264.


\bibitem{Ochiai00}
T. Ochiai,
\textit{Control theorem for Bloch-Kato's Selmer groups of $p$-adic representations.},
J. Number Theory 82 (2000), no. 1, 69-90.


\bibitem{Ochiai01}
T. Ochiai,
\textit{Control theorem for Greenberg's Selmer groups for Galois deformations.}, 
J. Number Theory 88 (2001), 59-85.






\bibitem{Ochiai06}
T. Ochiai,
\textit{On the two-variable Iwasawa main conjecture.},
Compos. Math. 142 (2006), no. 5, 1157-1200.



\bibitem{Ochiai08}
T. Ochiai,
\textit{The algebraic $p$-adic $L$-function and isogeny between families of Galois representations.},
J. Pure Appl. Algebra 212 (2008), no. 6, 1381-1393.

\bibitem{MO00}
M. Ohta,
\textit{Ordinary $p$-adic \'etale cohomology groups attached to towers of elliptic modular curves. II.},
Math. Ann. 318 (2000), no. 3, 557-583.


\bibitem{MO05}
M. Ohta,
\textit{Companion forms and the structure of $p$-adic Hecke algebras.},
J. Reine Angew. Math. 585 (2005), 141-172.

\bibitem{MO06}
M. Ohta,
\textit{Companion forms and the structure of $p$-adic Hecke algebras. $\mathrm{II}$.},
J. Math. Soc. Japan 59 (2007), no. 4, 913-951.



\bibitem{PR}
B. Perrin-Riou,
\textit{Variation de la fonction $L$ $p$-adique par isog\'enie.}, 
Algebraic number theory, 347-358, Adv. Stud. Pure Math., 17, Academic Press, Boston, MA, 1989.

\bibitem{Ri76}
K. Ribet,
\textit{A modular construction of unramified $p$-extensions of $\mathbb{Q}({\mu_p})$.},
Invent. Math. 34 (1976), no. 3, 151-162.



\bibitem{SeTe}
J. P. Serre, 
\textit{Trees.},
Translated from the French original by John Stillwell. Corrected 2nd printing of the 1980 English translation. Springer Monographs in Mathematics. Springer-Verlag, Berlin, 2003.

\bibitem{GS}
G. Stevens,
\textit{Stickelberger elements and modular parametrizations of elliptic curves.},
Invent. Math. 98 (1989), no. 1, 75-106.





\bibitem{Wi88}
A. Wiles, 
\textit{On ordinary $\lambda$-adic representations associated to modular forms.},
Invent. Math. 94 (1988), no. 3, 529-573.

\bibitem{Wi90}
A. Wiles, 
\textit{The Iwasawa conjecture for totally real fields.},
Ann. of Math. (2) 131 (1990), no. 3, 493-540. 


\bibitem{Y}
D. Yan,
\textit{Stable lattices in modular Galois representations and Hida deformation.},
J. Number theory 197 (2019), 62-88.

















\end{thebibliography}
\end{document}